\documentclass[12pt, reqno]{amsart}
\usepackage[utf8]{inputenc}

\usepackage{import}
\usepackage{graphicx}
\usepackage{color}
\usepackage{transparent}
\usepackage{calc}
\usepackage{float}

\usepackage{overpic}
\usepackage{amsmath,amsthm,amssymb,amscd,amsfonts,amsbsy}
\usepackage{color}
\usepackage{float}
\usepackage{hyperref}
\usepackage{enumerate}
\usepackage{enumitem}    
\usepackage{mathrsfs}
\usepackage{mathtools}
\usepackage{thmtools}
\usepackage{comment}
\usepackage{thm-restate}
\usepackage{cleveref}
\usepackage{graphics, setspace}
\usepackage{listings,xcolor}
\usepackage[normalem]{ulem}
\usepackage[toc,page]{appendix}
\usepackage[top=3cm,bottom=2cm,left=1.5cm,right=1.5cm]{geometry}
\usepackage{soul}

\allowdisplaybreaks

\newcommand{\mathsym}[1]{{}}
\newcommand{\unicode}[1]{{}}

\newcommand{\R}{\ensuremath{\mathbb{R}}}

\newcommand{\ov}{\overline}

\newcommand{\set}[1]{\left\{#1\right\}}
\newcommand{\hausdorff}[1]{\operatorname{dim_H}\left(#1\right)}
\newcommand{\cal}[1]{\mathcal{#1}}
\newcommand{\abs}[1]{\left\lvert{#1}\right\rvert}
\newcommand{\Bigabs}[1]{\Big|{#1}\Big|}
\newcommand{\norm}[1]{\left\lVert{#1}\right\rVert}

\newcommand{\cl}[1]{\ensuremath{\overline{#1}}}
\newcommand{\Dom}[1]{\ensuremath{\operatorname{Dom}(#1)}}
\newcommand{\proj}[0]{\ensuremath{\operatorname{proj}}}
\newcommand{\opname}[1]{\ensuremath{\operatorname{#1}}}

\newtheorem {theorem}{Theorem A}
\newtheorem {definition}[theorem]{Definition}

\newtheorem {proposition}[theorem]{Proposition}
\newtheorem {corollary}[theorem]{Corollary}

\newtheorem {remark}{Remark}

\newtheorem {mtheorem}{Theorem}

\def\R{\mathbb R}
\def\N{\mathbb N}

\renewcommand{\subset}{\subseteq}

\graphicspath{{figures/}}

\title[On the Hausdorff dimension and Cantor set structure of sliding Shilnikov invariant sets]{On the Hausdorff dimension and Cantor set structure\\ of sliding Shilnikov invariant sets}

\author[M. G. C. Cunha, D. D. Novaes, and G. Ponce]
{Matheus G. C. Cunha$^{1}$, Douglas D. Novaes$^1$, Gabriel Ponce$^1$}

\address{$^1$Departamento de Matem\'{a}tica - Instituto de Matem\'{a}tica, Estat\'{i}stica e Computa\c{c}\~{a}o Cient\'{i}fica (IMECC) - Universidade
Estadual de Campinas (UNICAMP), \ Rua S\'{e}rgio Buarque de Holanda, 651, Cidade Universit\'{a}ria Zeferino Vaz, 13083-859, Campinas, SP,
Brazil
}
\email{mathgccunha@ime.unicamp.br}
\email{ddnovaes@unicamp.br}
\email{gaponce@unicamp.br}

\begin{document}


\maketitle

\begin{abstract}
The concept of sliding Shilnikov connection has been recently introduced and represents an important notion in Filippov systems, because its existence implies chaotic behavior on an invariant subset of the system. The investigation of its properties has just begun, and understanding the topology and complexity of its invariant set is of interest. In this paper, we conduct a local analysis on the first return map associated to a sliding Shilnikov connection, which reveals a conformal iterated function system (CIFS) structure. By using the theory of CIFS, we estimate the Hausdorff dimension of the local invariant set of the first return map, showing, in particular, that is a positive number smaller than $1$, and with one-dimensional Lebesgue measure equal to $0$. Moreover, we prove that the closure of the local invariant set is a Cantor set and retains both the Hausdorff dimension and Lebesgue measure of the  local invariant set. Furthermore, this closure consists of the local invariant set along with the set of all pre-images, under the first return map, of the visible fold-regular point contained in the connection.
\end{abstract}

\section{Introduction}
\label{sec:introduction}

When modeling various phenomena, it is often observed that the rules governing their evolution change abruptly at specific thresholds, introducing discontinuities in their models (see \cite{jeffrey18,jeffrey20} for a general discussion on discontinuities in applied models). These abrupt changes are commonly associated with processes involving decisions, or switches, such as those in neurons or electronic systems, light refraction, body collisions, changes in dry friction regimes, or any other scenarios that exhibit sudden shifts in behavior. In ecology, for example, the phenomenon of prey-switching describes a predator's adaptive diet in response to the availability of different prey species. This behavior, observed in many predator species, creates discontinuities in prey-predator models (see, for example, \cite{Piltz2014,vanLeeuwen2013}). Similar discontinuities are found in other applied models, such as prey-predator models with prey refuge \cite{Kivan2011}, mechanical systems \cite{example4}, electromagnetic processes \cite{example1}, and others. The mathematical framework used to model and understand these phenomena includes the concept of piecewise smooth differential systems. Therefore, obtaining a better understanding of the geometric and dynamical properties of these systems is highly valuable. For surveys on piecewise smooth dynamical systems and their applications, see \cite{example3,example5}.

These types of differential systems, however, raise a fundamental question: what constitutes a solution? In \cite{filippov}, Filippov used the theory of differential inclusions to address this issue. He formulated what is now known as \emph{Filippov's convention} for the trajectories of piecewise smooth differential systems. Systems that follow this convention are referred to as \emph{Filippov systems} (for discussions on other conventions, see \cite{jeffrey14,JDYU22,NJ15}). The set of discontinuities of a Filippov system is called {\it switching set}, which, for our purposes, will always be assumed to be a smooth manifold.

Filippov's convention and its related concepts will be formally defined in the following subsections but, before that, let us briefly discuss in an informal manner this convention and the main object of our study. First, for points on the switching manifold where the vector fields in both sides cannot be concatenated  to create a trajectory that crosses the switching manifold, the Filippov's convention induces a dynamics on the switching manifold, allowing trajectories to slide along it, a phenomenon known as {\it sliding dynamics}. Also, the Filippov's convention extends the classical concept of singularities for smooth vector fields by inducing new types of singularities on the switching manifold. In these scenarios, the trajectories of a Filippov system may asymptotically approach these singularities or reach them in finite time, potentially leading to a loss of uniqueness of trajectories. The combination of these new singularities and the sliding dynamics gives rise to unique global phenomena in Filippov systems, such as the {\it sliding Shilnikov connection} (see Definition \ref{definition:shilnikov_connection}) that has been recently introduced and discussed in \cite{shilnikovproblem}.

The sliding Shilnikov connection is an important notion in Filippov systems, as their existence implies chaotic behavior within an invariant subset of the system, as demonstrated in \cite{npv}. This phenomenon has practical applications in applied science. For instance, in \cite{Piltz2014}, numerical evidences of chaotic behavior were observed in a prey-switching Filippov-type prey-predator model. This observation was analytically explained in \cite{predatorprey} by proving that such a model exhibits a sliding Shilnikov connection. The study of the properties and applicability of this connection is in its early stages, and understanding the topology and complexity of its associated invariant sets is of interest.

In the following subsections, we will formally discuss Filippov's convention and its related concepts, introduce the definition of sliding Shilnikov connection, and present our main findings.

\subsection{Filippov systems}

Let $ V \subset \mathbb{R}^n $ be an open subset, and consider the following piecewise smooth vector field:
\begin{equation}\label{psvf}
Z(u) = \begin{cases}
    X(u), & \text{if } g(u) > 0, \\
    Y(u), & \text{if } g(u) < 0,
\end{cases} \quad u \in V,
\end{equation}
where $ X, Y \in \mathcal{C}^{l, \varepsilon}(V, \mathbb{R}^n)$, with $l\geq 1$ an integer, $0<\varepsilon\leq1$, and $g \in \mathcal{C}^1(V, \mathbb{R})$ is a function with $0$ as a regular value (that is, $ D g(u) : \mathbb{R}^n \to \mathbb{R} $ is surjective for all $ u \in g^{-1}(0) $). Its switching manifold is given by $ M = g^{-1}(0)$. Recall that a function $ f : V \subset \mathbb{R}^n \to \mathbb{R}^m $ is said to be of class $ \mathcal{C}^{l, \varepsilon}(V, \mathbb{R}^m) $ if $ f \in \mathcal{C}^l(V, \mathbb{R}^m) $ and its $ l $-th derivative, $ D^l f $, satisfies the $ \varepsilon $-H\"{o}lder condition, that is, there exists a constant $ L $ such that $\norm{D^l f(x) - D^l f(y)}_{\opname{op}} \leq L   \abs{x - y}^{\varepsilon}$ for all $ x, y \in V $, where $\norm{\cdot}_{\opname{op}}$ is the standard norm among linear operators and $\abs{\cdot}$ is the usual euclidian norm on $\R^n$. When the context is clear, we may abbreviate this as $ \mathcal{C}^{l,\varepsilon}(V) $ or $ \mathcal{C}^{l, \varepsilon} $.

The piecewise smooth vector field \eqref{psvf} is concisely denoted as $ Z = (X,Y)_g $ (or simply $ Z = (X, Y) $). The space of all piecewise smooth systems of the form \eqref{psvf} is denoted by $ \Omega_g^{l, \varepsilon}(V, \mathbb{R}^n) \cong \mathcal{C}^{l, \varepsilon}(V, \mathbb{R}^n) \times \mathcal{C}^{l, \varepsilon}(V, \mathbb{R}^n) $, allowing us to endow it with the product topology. When the context is clear, we may abbreviate this as $ \Omega_g^{l, \varepsilon}(V) $, $ \Omega_g^{l, \varepsilon}, \Omega^{l, \varepsilon}(V) $, or simply $ \Omega^{l, \varepsilon} $.

The Filippov's convention establishes that the local trajectories of $Z$ (i.e. local solutions of the differential system $\dot{u} = Z(u)$) correspond to solutions of the differential inclusion
\begin{equation}
\label{equation:differential_inclusion}
\dot{u}\in\mathfrak{F}_Z(u),\,\, u\in V,
\end{equation}
where  $\mathfrak{F}_Z:V\leadsto\R^n$ is the following set-valued function
\begin{equation*}
\label{equation:local_trajectories}
\mathfrak{F}_Z(u) := \begin{cases}
	\set{X(u)}, &\quad\text{if } g(u)>0,\\
	\set{(1-s)X(u)+s\,Y(u) : t\in[0,s]}, &\quad\text{if } g(u)=0,\\
	\set{Y(u)}, &\quad\text{if } g(u)<0.
	\end{cases}
\end{equation*}
We recall that $\varphi:I\to V$, defined on an open interval $I\subset\R$, is said to be a solution of the differential inclusion \eqref{equation:differential_inclusion}, if it is an absolutely continuous function satisfying $\dot{\varphi}(t)\in\mathfrak{F}_Z(\varphi(t))$ for almost every $t\in I$. For an introduction on differential inclusions, see \cite{AC84}. 

The local trajectories of \eqref{psvf} have an intuitive geometric interpretation. To explore this, let $ \varphi^t_F(u) $ denote the flow of a vector field $ F: V \subset \mathbb{R}^n \to \mathbb{R}^n $  at time $ t $ starting from $ u $, and also define 
\begin{equation}
Fg(u) := \langle F(u), \nabla g(u) \rangle,
\end{equation}
where $\langle\cdot,\cdot\rangle$ denotes the usual inner product of $\R^n$. For points in $ V $ where $ g(u) \neq 0 $, the local trajectory corresponds to the local trajectories of either $ X $ or $ Y $, depending on whether $ g(u) > 0 $ or $ g(u) < 0 $, respectively. To describe the local solutions for points on the switching manifold $ M $, we first distinguish between some open regions on $ M $ (see Figure \ref{fig:filippov}):

\begin{itemize}
 \item The points on $ M$ satisfying $X g(u)  Y g(u) > 0$  define the \emph{crossing region $M^c$}. This implies that there exists $t_1<0<t_2$ such that $g(\varphi_X^{t}(u))>0$ for $t\in(t_1,0)$ and $g(\varphi_Y^{t}(u))<0$ for $t\in(0,t_2)$, or  $g(\varphi_X^{t}(u))>0$ for $t\in(0,t_2)$ and $g(\varphi_Y^{t}(u))<0$ for $t\in(t_1,0)$,  so that both trajectories can be concatenated at $u$ to form a local trajectory of $Z$ at $u$.

 \item The points on $M$ satisfying $X g(u)>0$ and $Y g(u)<0$ define the \emph{escaping region $M^e$}. This implies that there exists $t_2>0$ such that $g(\varphi_X^t(u)) > 0$ for $t \in (0,t_2)$ and $g(\varphi_Y^t(u)) < 0$, so that both trajectories also cannot be concatenated at $u$ to form a trajectory of $Z$.

    \item The points on $M$ satisfying $X g(u)<0$ and $Y g(u)>0$ define the \emph{sliding region $M^s$}. This implies that there exists $t_1<0$ such that $g(\varphi_X^{t}(u))>0$ for $t\in(t_1,0)$ and $g(\varphi_Y^{t}(u))<0$ for $t\in(t_1,0)$, so that both trajectories cannot be concatenated at $u$ to form a trajectory of $Z$.
    \end{itemize}

\begin{figure}[h!]
\centering
%\vspace{0.2cm}
\begin{overpic}[width=0.8\linewidth]{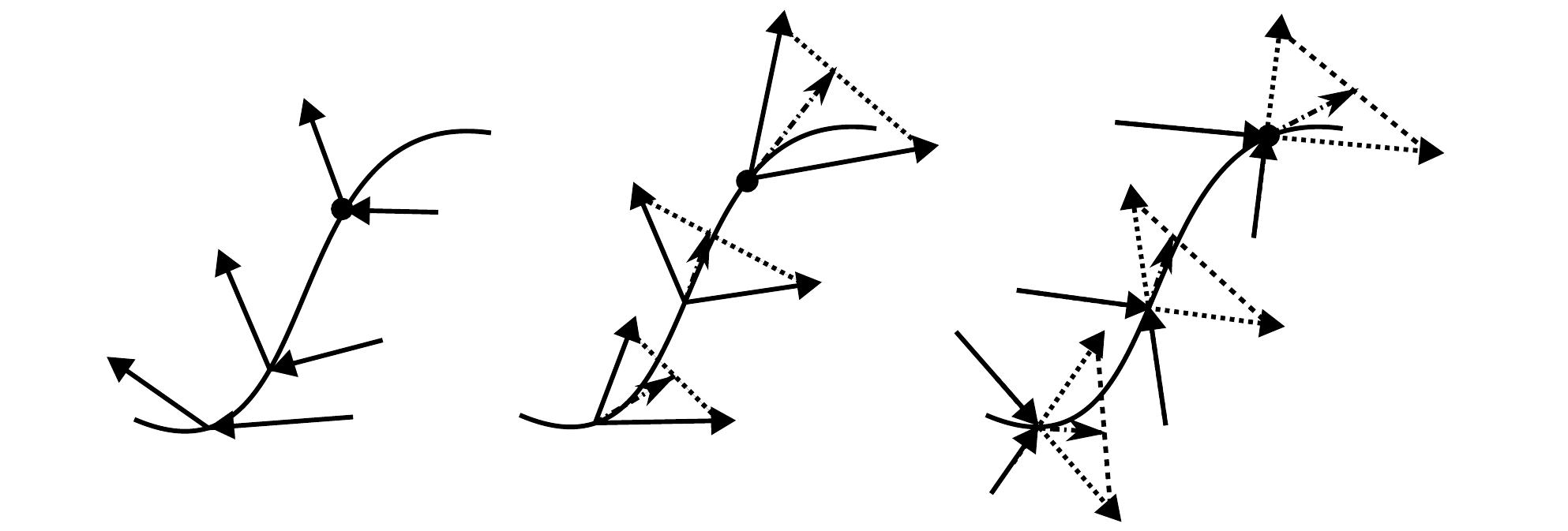}
%\begin{overpic}[grid,tics=2,width=0.8\linewidth]{filippov.pdf}
\put(21.5,18){$u$}
\put(47.5,20){$u$}
\put(79,26.5){$u$}
\put(3.5,6){$M^c$}
\put(28.5,6){$M^e$}
\put(58,6){$M^s$}
\put(17,29){$X(u)$}
\put(42,29){$X(u)$}
\put(72,27){$X(u)$}
\put(25,17){$Y(u)$}
\put(53,20){$Y(u)$}
\put(82,20){$Y(u)$}
\put(54,30.5){$\tilde{Z}(u)$}
\put(87,29){$\tilde{Z}(u)$}
\end{overpic}
%\vspace{0.2cm}
\caption{A graphical representation of the crossing, escaping and sliding regions on the switching manifold.}
\label{fig:filippov}
\end{figure}
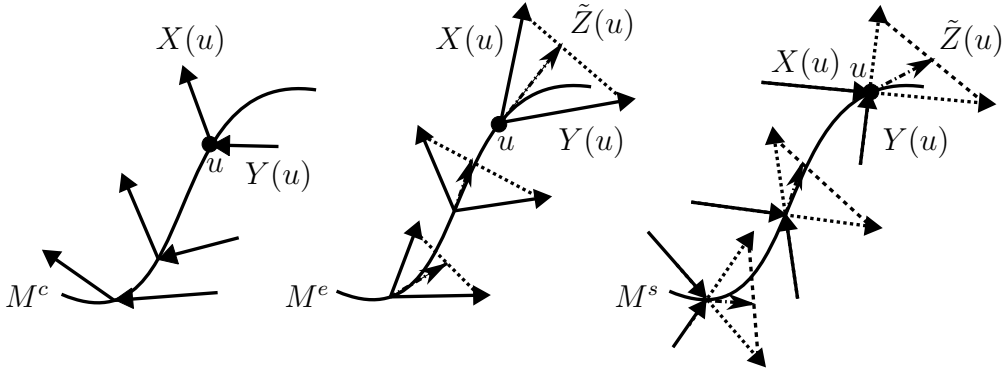

For a point $u\in M^c$, the local trajectory of \eqref{psvf} at $u$ is uniquely determined as a suitable concatenation of the local trajectories of $X$ and $Y$ at $u,$ as previously described.

For a point  $u \in M^{s,e} := M^{s} \cup M^{e}$, the local trajectories of  $X$  and  $Y$  at  $u$  cannot be concatenated, as they are either both approaching or both departing from  $M$ at  $u$. However, for each  $u \in M^{s,e}$, there exists a unique vector within the convex combination  $\mathfrak{F}_Z(u)$  that is tangent to  $M$  at  $u$. This vector is given by:
\begin{equation*}\label{sliding}
\tilde{Z}(u) := \dfrac{Y g(u)X(u) - X g(u)Y(u)}{Y g(u)-X g(u)}, \quad u\in M^{s,e}.
\end{equation*}
Since $\tilde{Z}(u) \in T_u M = T_u M^{s,e}$ for every $u \in M^{s,e}$, $\tilde{Z}(u)$ defines a vector field on $M^{s,e}$, referred to as the \textit{sliding vector field}. Additionally, because $\tilde{Z}(u) \in \mathfrak{F}_Z(u)$ for all $u \in M^{s,e}$, any local trajectory of $\tilde{Z}(u)$ satisfies the differential inclusion \eqref{equation:differential_inclusion} and, therefore, corresponds to a local trajectory of the Filippov vector field \eqref{psvf}. It should be noted that the local trajectories of $X$ or $Y$ at $u$ can eventually be concatenated with the local trajectory of $\tilde{Z}$ at $u$ to form additional local trajectories of $Z$ at $u$. Consequently, the uniqueness of local trajectories is not guaranteed for points in $M^{s,e}$.

The sliding dynamics, described above,  naturally introduces a first new type of singularity of the Filippov system $Z$, corresponding to the singularities of the sliding vector field $\tilde{Z}$. These are points $u^*\in M^{s,e}$ where $\tilde{Z}(u^*)=0$, known as \emph{pseudo-equilibria of $Z$}. Notably, trajectories of $\tilde{Z}$ can asymptotically approach a pseudo-equilibrium, while  trajectories of $X$ and $Y$ can either reach or depart from it in finite time. This type of singularity is crucial to the definition of sliding Shilnikov connections. A pseudo-equilibrium is considered hyperbolic if it is a hyperbolic singularity of $\tilde{Z}$. Additionally, if $u^* \in M^s$ is an unstable hyperbolic focus of $\tilde{Z}$, or if $u^* \in M^e$ is a stable hyperbolic focus of $\tilde{Z}$, then $u^*$ is referred to as a \emph{hyperbolic pseudo-saddle-focus} (this corresponds to the point $p$ in Figure \ref{fig:shilnikov_connection}).

Finally, we consider the set of \emph{tangency points} $M^t$, which consists of points $u \in M$ where $X g(u) Y g(u) = 0$. A point $u\in M^t$ is called a tangency point of $X$ if $X g(u) = 0$, or a tangency point of $Y$ if $Y g(u) = 0$. There are many possible configurations of points in $M^t$, leading to different definitions of their local trajectories. As a result, the uniqueness of local trajectories is also not guaranteed at points in $M^t$. In the following, we will introduce the concept of a {\it visible fold-regular point}, a specific type of tangency point that occurs in sliding Shilnikov connections.

A tangency point  $u \in M^t$  is referred to as a visible fold of  $X$  (resp.  $Y$) if  $X^2g(u) := X(Xg)(u) > 0$  (resp.  $Y^2g(u) := Y(Yg)(u) < 0$). Conversely, if the inequalities are reversed, the point $u$  is called an invisible fold of  $X$  (resp.  $Y$). A visible/invisible fold  $u \in M$  of  $X$  (resp.  $Y$) is called a visible/invisible fold-regular point if  $Yg(u) \neq 0$  (resp.  $Xg(u) \neq 0$). If  $Yg(u) > 0$ (resp.  $Xg(u) < 0$), the point lies on the boundary of the sliding region,  $\partial M^s$. Conversely, if  $Yg(u) < 0$  (resp.  $Xg(u) > 0$), it lies on the boundary of the escaping region,  $\partial M^e$. For instance, the point  $q$  in Figure \ref{fig:shilnikov_connection} corresponds to a visible fold-regular point lying on  $\partial M^s$.

\begin{remark}\label{fold-regular-properties}
Visible fold-regular points have several important properties (see \cite{TEIXEIRA199015}), of which we will highlight two.

\begin{enumerate}[label= (A\arabic*), ref = (A\arabic*)]
\item\label{remark:a1} Firstly, the sliding vector field $\tilde{Z}$ is always transverse to these points, which means that trajectories of $\tilde{Z}$ either reach or depart from them transversely within a finite time.

\item\label{remark:a2} Secondly, when the dimension of the space is greater than or equal to $3$ ($n \geq 3$), a visible fold-regular point $q$ is never isolated, which means that there exists a neighborhood $U$ around $q$ such that $U \cap \partial M^{s,e}$ is a set consisting entirely of visible fold-regular points.
\end{enumerate}
\end{remark}

As a consequence of the above property \ref{remark:a1}, the local trajectory of $Z$ at a visible fold-regular point $u$ is determined by an appropriate combination of local trajectories of $X$, $Y$, and $\tilde{Z}$ at $u$.

\subsection{Sliding Shilnikov connection} The concept of sliding Shilnikov connection was recently introduced in \cite{shilnikovproblem}. Roughly speaking, it consists of a trajectory $\Gamma$ of $Z$, passing though a visible fold-regular point $q\in\partial M^{s,e}$ and connecting a hyperbolic pseudo-saddle-focus $p$ to itself asymptotically on one side and in finite time on the other side (see Figure \ref{fig:shilnikov_connection}).

Some of its dynamical properties, such as chaotic behavior, were further explored in \cite{npv}. It has also demonstrated significant applied importance, as shown in \cite{predatorprey}, by proving that a family of prey-switching Filippov-type prey-predator models exhibits chaotic behavior, which had previously been supported only by numerical evidence in \cite{Piltz2014}.

In what follows, we introduce the definition of sliding Shilnikov connection.

\begin{definition}[Sliding Shilnikov Connection]
\label{definition:shilnikov_connection}
Let $Z=(X,Y)\in\Omega^{1,\varepsilon}$ be a Filippov system with a hyperbolic pseudo-saddle-focus $p\in M^s$ (resp. $p\in M^e$) and a visible fold-regular point $q\in \partial M^s$ (resp. $q\in \partial M^e$), which is a visible fold point of $X$. Assume that:\\
\begin{enumerate}
    \item The trajectory of $\tilde{Z}$ passing through $q$ converges to $p$ backward in time (resp. forward in time), that is, $\lim_{t\to-\infty}\varphi^{t}_{\Tilde{Z}}(q) = p$ (resp. $\lim_{t\to+\infty}\varphi^{t}_{\Tilde{Z}}(q) = p$).\\
    \item The trajectory of $X$ passing through $q$ reaches $M^s$ (resp. $M^e$) in a finite time $t_q > 0$ (resp. $t_q<0)$ at $p$, that is, $\varphi^{t_q}_{X}(q)=p$ and $g(\varphi^{t}_X(q)) \neq 0$, for any $t\in(0,t_q)$ (resp. $t\in(t_q,0)$).\\
\end{enumerate}
Then, the sliding loop $\Gamma$ passing through $q$ and connecting $p$ to itself is called a \emph{sliding Shilnikov connection} (see Figure \ref{fig:shilnikov_connection}).
\end{definition}

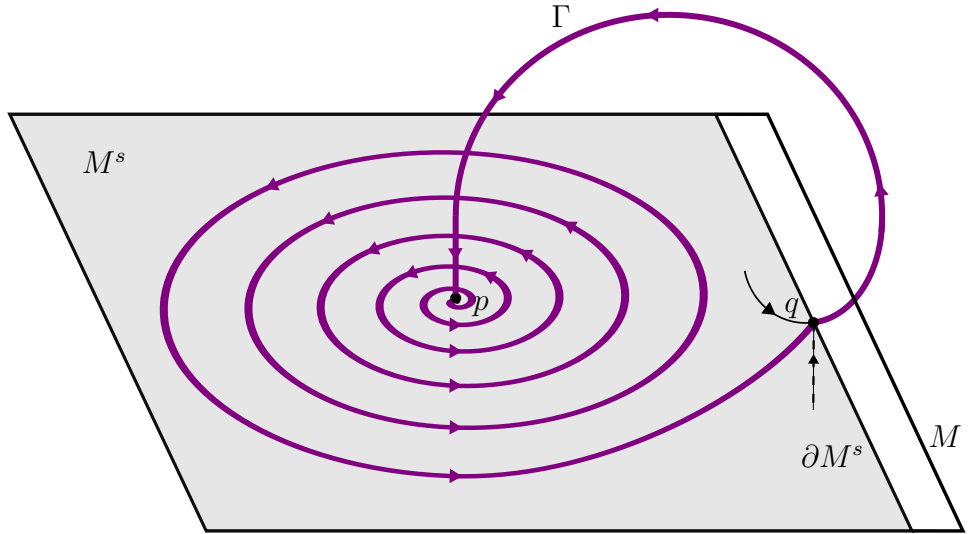
\begin{figure}[h!]
\centering 
\begin{overpic}[width=0.7\linewidth]{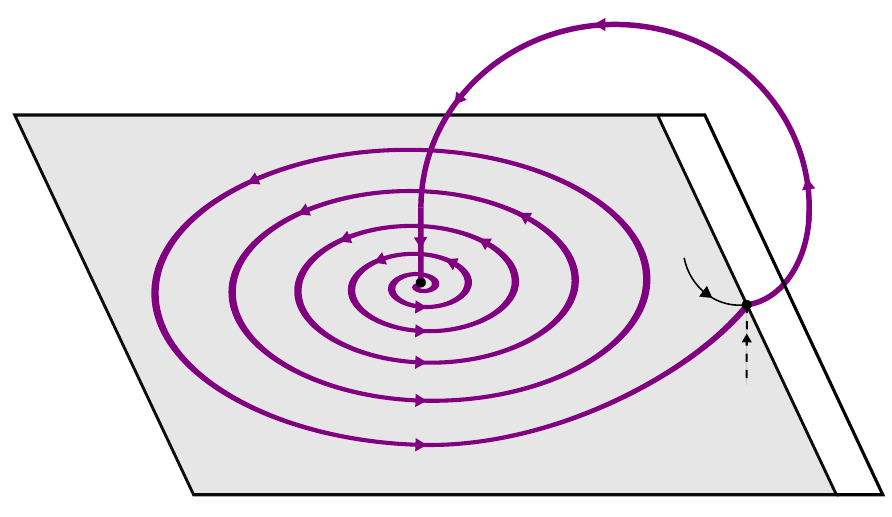}
%\begin{overpic}[grid,tics=5,width=0.7\linewidth]{shilnikov_connection.pdf}
\put(57,54){$\Gamma$}
\put(49,25){$p$}
\put(80.8,24.8){$q$}
\put(9,39){$M^s$}
\put(95.5,11){$M$}
\put(82.5,9){$\partial M^s$}
\end{overpic}
%\vspace{0.2cm}
\caption{Representation of a sliding Shilnikov connection. The point $p\in M^s$ is a hyperbolic pseudo-saddle-focus and the point $q\in \partial M^s$ is a visible fold-regular point for $X$. The forward trajectory of $Z$ at $q$ follows the flow of the vector field $X$ until it reaches, in finite time, the sliding region $M^s$ at the hyperbolic pseudo-saddle-focus $p$. The backward trajectory of $Z$ at $q$ follows the backward flow of the sliding vector field $\tilde{Z}$ which approaches asymptotically to the hyperbolic pseudo-saddle-focus $p$.}
\label{fig:shilnikov_connection}
\end{figure}

In \cite{npv}, it was demonstrated the existence of a neighborhood $B\subset\R^3$ of $q$ such that, for $\gamma:=B\cap\partial M^{s,e}$, which is a curve of visible fold-regular points (see \ref{remark:a2} from Remark \ref{fold-regular-properties}), a first return map  $\pi: \Dom{\pi} \to \gamma$ is well-defined on a subset $\text{Dom}(\pi)\subset \gamma$. This map captures the complete dynamics of the Filippov system $Z$ in a neighborhood of the sliding Shilnikov connection $\Gamma$. The dynamics of $\pi$ exhibits a rich structure with many interesting properties. For instance, in \cite{npv}, it was shown that the restriction of $\pi$ to an invariant set is topologically conjugate to a Bernoulli shift with infinite topological entropy. In \Cref{subsec:firstreturn}, we will provide details on the construction of the first return map $\pi$.

The main goal of this study is to examine, from a local perspective, certain geometric properties of the invariant set of the first return map restricted to $U$, $\pi_U:=\pi|_{\textrm{Dom}(\pi)\cap U}$, where $U \subset \gamma$ is a neighborhood of $q$. Such invariant set is given by
\begin{equation}
\label{equation:invariantset}
\Lambda_U := \set{w\in U: \pi^k(w)\in \textrm{Dom}(\pi)\cap U, \text{ for all } k\geq0}.
\end{equation}
More specifically, our goal is to estimate the Hausdorff dimension, denoted by $\hausdorff{\cdot}$, and the one-dimensional Lebesgue measure, denoted by $\opname{m}_1(\cdot)$, as well as to explore the Cantor set topological structure for both the invariant set and its closure. 

As we will show, the first return map can be decomposed into maps exhibiting expanding behavior, such that the functions corresponding to their inverses are contractions (see Figure \ref{fig:cifs}). This brings us into the realm of \emph{Iterated Function Systems (IFS)}, or more specifically, \emph{Conformal Iterated Function Systems (CIFS)} when the functions meet suitable criteria. We will use tools from IFS theory to analyze the geometric aspects of $\Lambda_U$ mentioned before (see \Cref{sec:cifs}).

\subsection{Main result}
Our main findings can be summarized as follows.

\begin{mtheorem}\label{theorem}
Let $V\subset\R^3$ be an open subset and consider a Filippov system $Z = (X,Y)\in\Omega_g^{1,1}(V, \R^3)$ possessing a sliding Shilnikov connection $\Gamma$ passing through a visible fold-regular point $q\in\partial M^{s,e}$. Consider the first return map $\pi:\Dom{\pi}\subset\gamma \to \gamma$ associated to $\Gamma$. Then, there exists a neighborhood $U\subset\gamma$ of $q$ such that:
\begin{enumerate}[label= (\alph*), ref = Statement (\alph*)]
	\item\label{theorem:hausdorff_and_lebesgue}
	The invariant set $\Lambda_U$ satisfies $0<\hausdorff{\Lambda_U} < 1$ and $\opname{m}_1(\Lambda_U)=0$. 
	\item\label{theorem:closure_of_lambda}
The closure of the invariant set, $\ov{\Lambda_U},$ is a Cantor set,  $\hausdorff{\ov{\Lambda_U}}=\hausdorff{\Lambda_U}$, and  $\opname{m}_1(\ov{\Lambda_U})=\opname{m}_1(\Lambda_U)$. Furthermore, $\ov{\Lambda_U}=\Lambda_U\ \dot{\cup}\ Q_U$, where $Q_U = \left(\bigcup_{k\geq0}\pi^{-k}(q)\right)\cap U$ is a countable set.
\end{enumerate}
\end{mtheorem}

Recall that a metric space $K \neq \emptyset$  is called a \emph{Cantor set} if it is compact, totally disconnected (the only connected subsets are singletons), and perfect (every point is an accumulation point). It is well-known that all non-empty metric spaces with these topological properties are homeomorphic to each other (see, for example, \cite[Chapter 2, Theorem 67 and Theorem 73]{cantor_spaces}).

The proof of Theorem \ref{theorem} is made in Section \ref{sec:proof}, and is based on iterated function systems theory, which is introduced in  \Cref{basic}.

\section{Hausdorff dimension of attractor sets of iterated function systems}\label{basic}
This section introduces essential concepts and results required to prove Theorem \ref{theorem}.  We begin with the definition of the Hausdorff dimension and discuss some of its basic properties. Following this, we explore the concept of iterated function systems and the conditions under which they are conformal. Finally, we present results for estimating the Hausdorff dimension of their attractor sets.

\subsection{Hausdorff Dimension}
\label{sec:hausdorff}
The Hausdorff dimension is a notion of dimension related to the ``size'', in a measure sense, of a set. More precisely, let us consider the diameter of an arbitrary subset $V\subset \R^d$,
$$\opname{diam}(V) := \sup\set{\abs{x-y}:x,y\in V}.$$
Given a set $S$ and a real number $s\geq0$, for each $\delta > 0$ we look at all countable (possibly finite) covers of $S$ with subsets of diameter at most $\delta$ (also called a \emph{$\delta$-cover}) and calculate the quantity
$$\cal{H}^s_{\delta}(S) := \inf\set{\sum_{k\geq1}\opname{diam}(V_k)^s : \text{ $\set{V_k}_{k\geq1}$ is a $\delta$-cover of $S$}}.$$
We define the \emph{$s$-dimensional Hausdorff measure} of $S$ as the limit
$$\cal{H}^s(S) := \lim_{\delta\to0}\cal{H}^s_{\delta}(S).$$
The \emph{Hausdorff dimension} of $S$, denoted by $\hausdorff{S}$, is then determined by 
$$\hausdorff{S} := \sup\set{s \geq 0 : \cal{H}^s(S) = +\infty} = \inf\set{s \geq 0 : \cal{H}^s(S) =  0},$$
which is well-defined.

Important properties of the Hausdorff dimension are discussed in \cite{hausdorff_dimension} and \cite[Chapter 2]{falconer}. The properties that will be relevant to our discussion are listed below:

\begin{enumerate}[label = (B\arabic*), ref = Property (B\arabic*)]
    \item \label{hausdorffprop:monotonicity}
    If $S_i\subset S_j$, then $\hausdorff{S_i} \leq \hausdorff{S_j}$ (\cite[Theorem 2(1)]{hausdorff_dimension}).\\
    \item \label{hausdorffprop:stability}
    If $\set{S_k}_{k\geq1}$ is a countable collection of sets, then $\hausdorff{\bigcup_{k\geq1}S_k} = \sup_{k\geq1}\set{\hausdorff{S_k}}$ (\cite[Theorem 2(2)]{hausdorff_dimension}).\\
    %\item\label{hausdorffprop:countable_set}
    %(\cite[Theorem 2.3]{hausdorff_dimension}) If $S$ is a countable set, then we have $\hausdorff{S} = 0$.\\
    \item \label{hausdorffprop:invariance}
    If $f$ is a bi-Lipschitz map, then $\hausdorff{S} = \hausdorff{f(S)}$, which implies that diffeomorphic compact sets have the same Hausdorff dimension (\cite[Theorem 2(5)]{hausdorff_dimension}).\\
    \item \label{hausdorffprop:lebesgue_zero}
    If $\hausdorff{S} < 1$, then the one-dimensional Lebesgue measure of $S$ is $0$ (\cite[Theorem 2(8)]{hausdorff_dimension}).\\
    \item \label{hausdorffprop:totally_disconnected}
    If $\hausdorff{S} < 1$, then $S$ is totally disconnected (\cite[Proposition 2.5]{falconer}).\\
\end{enumerate}

\subsection{Iterated Function Systems (IFS)}

For this section, we assume that $K \subset \R^d$ is a compact subset satisfying $K = \cl{\opname{Int}(K)}$, and that $I$ is a countable (finite or infinite) index set.

An \emph{iterated function system (IFS) on $K$} is a family of contractions on $K$, $\mathcal{F} = \set{f_i:K\to K}_{i\in I}$, that is, a set of functions such that, for each $i\in I$, there exists $c_i<1$ for which
\[|f_i(x)-f_i(y)|\leq c_i |x-y|, \quad \text{for all } x,y\in K.\]

Given an IFS $\mathcal{F} := \set{f_i}_{i\in I}$, we consider finite words of $I$ of the form $\eta = (\eta_1,\ldots,\eta_k)\in  \bigcup_{j\geq1} I^j$, $k\geq 1$. For such a word $\eta$, we define the corresponding function $f_{\eta}$ as
$$f_{\eta} = f_{(\eta_1,\ldots,\eta_k)} := f_{\eta_1}\circ\cdots\circ f_{\eta_k}.$$
A key concept concerning an IFS is its {\it attractor set}. To define it, let $\proj: I^{\N} \to 2^K$ be the projection map from the symbol space $I^{\N}$ into $K$, that is, for each $\omega = (\omega_1, \omega_2, \ldots) \in I^{\N}$, $\proj(\omega)$ is defined satisfying
\begin{equation}\label{proj}
\proj(\omega) = \bigcap_{k\geq1}f_{\omega|_k}(K),
\end{equation}
where $\omega|_k := (\omega_1, \ldots, \omega_k)$, for $k\geq 1$.

Then, the attractor set of the IFS is defined as
\begin{equation}\label{atractset}
\Delta := \proj\left(I^{\N}\right) = \bigcup_{\omega\in I^{\N}} \bigcap_{k\geq1}f_{\omega|_k}(K).
\end{equation}

This set satisfies the self-similarity property, that is
$\Delta = \bigcup_{i\in I}f_i(\Delta)$. When the index set $I$ is finite, there is only one non-empty set that satisfies this property, which may not be the case when $I$ is infinite. However, the attractor set is the largest set that satisfies the self-similarity property (see \cite[Remark of Section 3]{mauldin}).

\begin{remark}\label{rem:proj}
It is worth noting that if there exists $0<s<1$ for which the contraction constants satisfy $c_i \leq s$, the set in \eqref{proj} becomes a singleton for each $\omega \in I^{\mathbb{N}}$. In this case, $\proj$ can also be regarded as a function from $I^{\mathbb{N}}$ into $K$, in a minor abuse of notation.
\end{remark}

Also, we should note that, if $\mathcal{F}'\subset\mathcal{F}$ are IFS (we say that $\mathcal{F}'$ is a sub-system of $\mathcal{F}$), then their respective attractor sets, namely $\Delta'$ and $\Delta$, satisfy $\Delta'\subset\Delta$.

Certain results on the Hausdorff dimension of the attractor set of an IFS provide bounds in terms of the contraction constants. For a finite IFS, we have the following result:

\begin{proposition}[{\cite[Propositions 9.6 and 9.7]{falconer}}]
\label{prop:calculation}

Let $\set{f_1, \ldots, f_k}$ be an IFS on $K$, with attractor set $\Delta$ and satisfying the following condition: given any $1\leq i\leq k$, there exists $0<b_i\leq c_i<1$ such that 
\[b_i|x-y| \leq |f_i(x)-f_i(y)| \leq c_i|x-y|, \quad \text{for all } x,y\in K.\] 
Besides, let us assume that $f_i(\Delta) \cap f_j(\Delta) = \emptyset$, for every $i\neq j$. Then, 
\[s \leq \hausdorff{\Delta} \leq t,\]
where $s$ and $t$ are the unique real numbers satisfying 
\[\sum_{1\leq i\leq k} b_i^s = 1 \quad \text{and}\quad  \sum_{1\leq i\leq k} c_i^t = 1.\]
\end{proposition}

As a consequence, we obtain the following lower bound for the Hausdorff dimension of the attractor set of a non-trivial IFS.

\begin{corollary}
\label{prop:hausdorffzero}

Let $\mathcal{F} = \set{f_i}_{i\in I}$ be an IFS having at least two different functions, $f_{i_1}$ and $f_{i_2}$, satisfying $b\abs{x-y} \leq \abs{f_{i_j}(x)-f_{i_j}(y)}$, for some $b>0$ and $j=1,2$. Then $\hausdorff{\Delta}>0$.
\end{corollary}

\begin{proof}
Let us consider the sub-system $\mathcal{F}' = \set{f_{i_1}, f_{i_2}} \subset \mathcal{F}$. We can then apply \Cref{prop:calculation} and \ref{hausdorffprop:monotonicity} on the attractor set $\Delta'$ and note that $s\leq\hausdorff{\Delta'} \leq \hausdorff{\Delta}$, where $s$ is the non-negative number that satisfies $b^s + b^s = 1$. This number has to be greater than $0$, since $b^0+b^0=2\neq1$. Therefore, we have $0 < s \leq \hausdorff{\Delta'} \leq \hausdorff{\Delta}$.
\end{proof}

\subsection{Conformal Iterated Function Systems (CIFS)}
\label{sec:cifs}

The Hausdorff dimension of the attractor set of an IFS cannot be estimated as quite as simple when the index set $I$ is infinite. Nevertheless, if an (infinite) IFS satisfies the so-called \emph{conformal conditions}, we have some results that allow us to consider approximations by finite IFS.

An IFS $\cal{F}$ is said to be a \emph{conformal IFS}, or just a \emph{CIFS}, if it satisfies the following conditions:

\begin{enumerate}[label= (C\arabic*), ref = (C\arabic*)]
    \item \label{conf:injection}
    For each $i\in I$, the function $f_i$ is an injection of $K$ into itself.\\
    
    \item \label{conf:unifcontract}
    The system is uniformly contractive on $K$, that is, there exists some $s<1$ such that
    $$\abs{f_i(x) - f_i(y)} \leq s\abs{x - y}, \quad \text{for all }i\in I \text{ and for all } x,y\in K.$$
    
    \item \label{conf:osc}
    (Open Set Condition) The set $K$ is connected, and each function satisfies $f_i(\opname{Int}(K))\subset \opname{Int}(K)$ and $f_i(\opname{Int}(K))\cap f_j(\opname{Int}(K))=\emptyset$, for all $i,j\in I, i\neq j$.\\
    
    \item \label{conf:extension}
    There is an open set $V\subset\R^d$, with $K\subset V$, such that each $f_i$ extends to $f_i^V$, a $\cal{C}^{1,\varepsilon}$ diffeomorphism on $V$, and the extensions are conformal functions, that is, the derivatives $Df_i^V(x)$ satisfy $Df_i^V(x) = \kappa_{x,i}\opname{Isom}_{x,i}$, where $\kappa_{x,i}\in\R$ and $\opname{Isom}_{x,i}:\R^d\to\R^d$ is an isometry, for any $x\in V$ and $i\in I$ (for more details on conformal functions, see \cite[Chapter 7]{pesin_dimension}).\\

	\item\label{conf:cone_alt}
	The following inequality holds: $$\inf_{x\in\partial K}\inf_{0<r<1}\dfrac{\opname{m}_d(B_r(x)\cap\opname{Int}(K))}{\opname{m}_d(B_r(x))} >0,$$ where $\opname{m}_d$ is the $d$-dimensional Lebesgue measure.\\

    \item\label{conf:bdp_alt}
    There are constants $L \geq 1$ and $\alpha > 0$ such that, for every $i\in I$,
    $$\Bigabs{\norm{Df_i^V(x)}_{\opname{op}} - \norm{Df_i^V(y)}_{\opname{op}}} \leq L \cdot \norm{(Df_i^V)^{-1}}_{\opname{unif}}^{-1} \cdot \abs{x-y}^{\alpha},$$
    where $f_i^V$, $i\in I$, are the extensions established in \ref{conf:extension}, $x,y\in V$, $\norm{\cdot}_{\opname{op}}$ is the operator norm, and $\norm{(Df_i^V)^{-1}}_{\opname{unif}} := \sup_{z\in V}\set{\norm{(Df_i^V(z))^{-1}}_{\opname{op}}}$.\end{enumerate}

It is important to note that \ref{conf:cone_alt} and \ref{conf:bdp_alt} are alternative formulations to those traditionally stated, such as in \cite{mauldin}. These alternative conditions are more suitable for our purposes and are discussed, for instance, in \cite[Theorem 3.2]{mauldin}, \cite[Lemma 2.2]{mu}, and \cite{gmw}.

The following proposition establishes a relationship between the Hausdorff dimension of the attractor set of a CIFS and its finite approximations.

\begin{proposition}[{\cite[Theorem 3.15]{mu}}]
\label{prop:sup}
Let $\mathcal{F} = \set{f_i}_{i\in I}$ be a CIFS with attractor set $\Delta$. Then, the equality $$\hausdorff{\Delta} = \sup_{J\in\opname{Fin}(I)}\set{\hausdorff{\Delta_J} : \Delta_J \text{ is the attractor set of the sub-system }\mathcal{F}_J=\set{f_j}_{j\in J}}$$ holds, where $\opname{Fin}(I) \subset 2^I$ is the collection of finite subsets of $I$.
\end{proposition}

For an analysis of the Hausdorff dimension of attractor sets and measures of CIFS, see also \cite{mihailescu}.

Now we present another proposition concerning CIFS, which states that in the definition of the attractor set \eqref{atractset}, the union and intersection can be, in a sense, interchanged.

\begin{proposition}[{\cite[Theorem 3.1]{mauldin}}]
\label{proposition:switch_union_intersection}
Let $\mathcal{F} = \set{f_i : i\in I}$ be a CIFS. Then, the following relationship holds:
$$\Delta = \bigcup_{\omega\in I^{\N}} \bigcap_{k\geq1}f_{\omega|_k}(K) = \bigcap_{k\geq1}\bigcup_{\eta\in I^k}f_{\eta}(K).$$
\end{proposition}

\section{Proof of the main result}
\label{sec:proof}
The local dynamics of a sliding Shilnikov connection is captured by the first return map $\pi$, defined on a subset of $\partial M^{s,e}$ within a small neighborhood $U \subset \partial M^{s,e}$ of the visible fold-regular point $q$. Thus, our initial task is to construct this map. Subsequently, based on its construction, we will relate this map to  the theory of CIFS by showing that it can be branched into countably many smooth maps such that the functions corresponding to their inverses constitute a CIFS.

For clarity, we assume the hypotheses of \Cref{theorem} in this section. Without loss of generality, we will focus exclusively on the sliding case, where $p \in M^s$ and $q \in \partial M^{s}$, as depicted in \Cref{fig:shilnikov_connection}.

\subsection{Construction of the First Return Map}
\label{subsec:firstreturn}

Let us consider a sliding Shilnikov connection $\Gamma$ (see Definition \ref{definition:shilnikov_connection}) with a hyperbolic pseudo-saddle-focus $p \in M^s$ and passing through a visible fold-regular point $q\in\partial M^s$, which is a visible fold point of $X$. Recall that, for $u\in V\subset\R^3$ and $w\in\cl{M^s}$, the flows of $X$ and $\tilde{Z}$ are denoted, respectively, by $\varphi^{t}_{X}(u)$ and $\varphi^{t}_{\tilde{Z}}(w)$.

First, for $ r > 0 $, let us define the set
$$\gamma_r := \overline{B_r(q) \cap \partial M^s}.$$ 
\ref{remark:a2} from Remark \ref{fold-regular-properties} implies that, for sufficiently small $ r > 0 $, $\gamma_r$ is a smooth curve of visible fold-regular points of $ Z $, which are also visible fold points for $ X $. In addition, \ref{remark:a1} from Remark \ref{fold-regular-properties} states that $\gamma_r$ is a transversal section of the sliding vector field $\tilde{Z}$.

Second, from Definition \ref{definition:shilnikov_connection}, there exists $t_q > 0$ such that $\varphi^{t_q}_{X}(q)=p$, in particular, $g(\varphi^{t_q}_{X}(q))=0,$ and $g(\varphi^{t}_X(q)) \neq 0$, for any $t\in(0,t_q)$. Since
\[
\dfrac{d}{dt}g(\varphi^{t}_{X}(q))\big|_{t=t_q}=Xg(p)\neq0,
\]
the Implicit Function Theorem implies the existence of $r>0$, and a function $t_X(w)$, defined in $\gamma_r$, satisfying
$$t_X(q)=t_q,\,\, g\big(\varphi^{t_X(w)}_X(w)\big) = 0, \text{ and } g\left(\varphi^{t}_X(w)\right) \neq 0\text{ for all }t\in(0,t_X(w)).$$
Thus, define
$$\mu_r := \set{\varphi^{t_X(w)}_{X}(w) : w\in\gamma_r}.$$
Notice that the flow of the vector field $X$ maps $\gamma_r$ onto $\mu_r$ by means of the following diffeomorphism
\begin{equation}\label{difeotheta}\begin{aligned}
\theta_X:\gamma_r&\to \mu_r\\
                w&\mapsto\theta_X(w) := \varphi^{t_X(w)}_{X}(w),
\end{aligned}\end{equation}
implying that $\mu_r$ is a smooth curve containing $p$. Furthermore, since $p \in \mu_r \subset M^s$ is a hyperbolic focus of $\tilde{Z}$, we can see that $\mu_r \setminus \set{p}$ is transversal to $\tilde{Z}$. Indeed, if this was not the case, the vector $0\neq v \in T_p M$ tangent to $\mu_r$ at $p$ would be a real eigenvector of $D\tilde{Z}(p)$ restricted to $T_p M$. However, this is impossible because the restriction of $D\tilde{Z}(p)$ to $T_p M$ has a pair of complex conjugate eigenvalues.

Third, consider the backward saturation, $S_r,$ of $\gamma_r$ induced by $\varphi_{\tilde{Z}}$, that is, 
$$S_r := \bigcup_{t\geq0}\varphi^{-t}_{\tilde{Z}}(\gamma_r).$$ 
From Definition \ref{definition:shilnikov_connection}, we have $\lim_{t \to -\infty} \varphi^t_{\tilde{Z}}(q) = p$. Therefore, taking into account that $\gamma_r$ is a transversal section of $\tilde Z$ and that $p\in\mu_r\subset M^s$ is a hyperbolic focus of $\tilde{Z}$, the Implicit Function Theorem can be used to ensure that, for sufficiently small $r > 0$, $\varphi^t_{\tilde{Z}}(w)$ converges to $p$ as $t \to -\infty$ for all $w \in \gamma_r$.
Now, since $\mu_r\setminus\{p\}$ is transversal to $\tilde Z$ and $p\in\mu_r\subset M^s$ is a hyperbolic focus of $\tilde{Z}$, we deduce that
$$S_r\cap\mu_r = \bigcup_{i\geq0}\left(L_i^{\mu_r}\cup R_i^{\mu_r}\right),$$
where the collections $\set{L_i^{\mu_r}}_{i\geq0},\set{R_i^{\mu_r}}_{i\geq0}\subset\mu_r$, of the connected components of $S_r\cap\mu_r$, are located in opposite sides of $p\in\mu_r$, $L_i^{\mu_r}\cap L_j^{\mu_r} = \emptyset = R_i^{\mu_r}\cap R_j^{\mu_r}$, if $i\neq j$, and $L_i^{\mu_r},R_i^{\mu_r}$ converge to $\set{p}$ in the Hausdorff distance as $i$ increases (see Figure \ref{fig:intervals}). In addition, for each $\zeta\in S_r\cap\mu_r $, there exists $t_{\tilde{Z}}(\zeta)>0$ satisfying $\varphi^{t_{\tilde{Z}}(\zeta)}_{\tilde{Z}}(\zeta) \in \gamma_r$ and $\varphi^{t}_{\tilde{Z}}(\zeta) \in M^s$, for all $t\in(0,t_{\tilde{Z}}(\zeta))$. This induces the following maps from each  $J\in\{L_i^{\mu_r}\}_{i\geq 0}\cup\{R_i^{\mu_r}\}_{i\geq 0}$  into  $\gamma_r$,
\begin{equation}\label{rhoZJ}
\begin{aligned}
\theta_{\tilde Z}^{J}:J&\to \gamma_r\\
        \zeta&\mapsto\pi(w) := \varphi^{t_{\tilde{Z}}(\zeta)}_{\tilde{Z}}(\zeta).\end{aligned}
\end{equation}
For $J\in\{L_i^{\mu_r}\}_{i\geq 1}\cup\{R_i^{\mu_r}\}_{i\geq 1}$, such maps are diffeomorphisms. However, for  $J\in\{L_0^{\mu_r},R_0^{\mu_r}\}$, the induced map $\theta_{\tilde Z}^{J}$ may not be surjective (see Figure \ref{fig:intervals}).

\begin{figure}[h!]
\centering
\begin{overpic}[width=\linewidth]{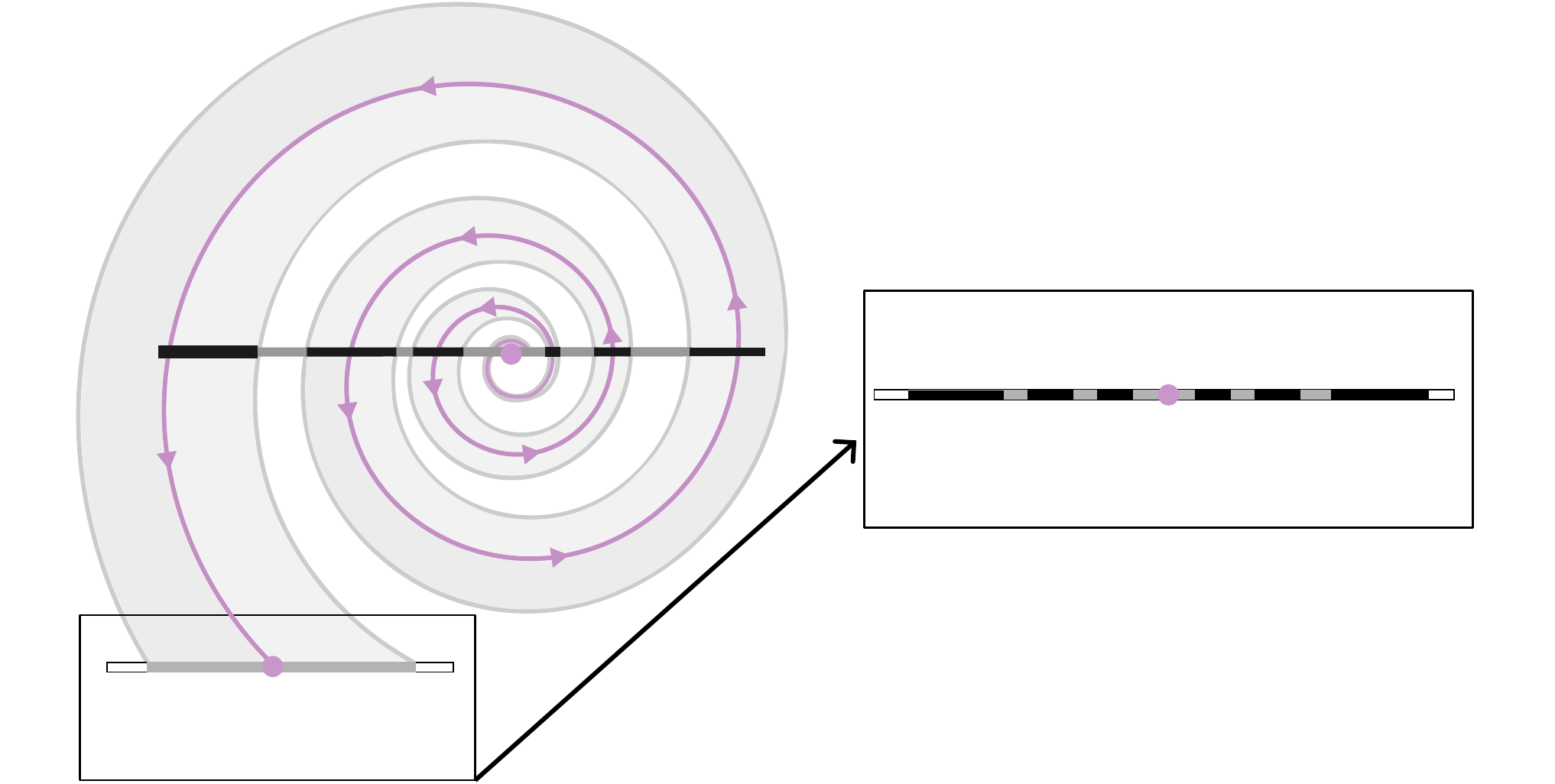}
%\begin{overpic}[grid,tics=2,width=\linewidth]{intervals.pdf}
\put(18,5.5){$q$}
\put(9.45,5){$\underbrace{\quad\quad\quad\quad\quad\quad\quad\quad}_{%
\let\scriptstyle\textstyle\substack{\gamma_r}}$}
\put(33.5,26){$p$}
\put(8,14){$S_r$}
\put(48,26){$\mu_r$}
\put(13,29){$L_0^{\mu_r}$}
\put(20.2,29){$L_1^{\mu_r}$}
\put(26,29){$L_2^{\mu_r}$}
\put(31.8,29){{\small $\cdots$}}
\put(35,29){$R_2^{\mu_r}$}
\put(39.7,29){$R_1^{\mu_r}$}
\put(47.7,29){$R_0^{\mu_r}$}
\put(60,26.5){$L_0^{\gamma_r}$}
\put(66,26.5){$L_1^{\gamma_r}$}
\put(71,26.5){$L_2^{\gamma_r}$}
\put(74,26.5){{\footnotesize $\cdots$}}
\put(77,26.5){$R_2^{\gamma_r}$}
\put(81,26.5){$R_1^{\gamma_r}$}
\put(88,26.5){$R_0^{\gamma_r}$}
\put(58.6,22.2){$\underbrace{\quad\quad\quad\quad\quad\quad\quad\quad\quad\quad\quad\quad\quad\quad\quad}_{
\let\scriptstyle\textstyle\substack{\gamma_r}}$}
\put(75.5,23){$q$}
\end{overpic}
\caption{On the left, we have a representation of the sliding dynamics, illustrating the intersection between the backward saturation $ S_r $ and $ \mu_r $, which generates the collections $ \{L^{\mu_r}_i\}_{i \geq 0} $ and $ \{R^{\mu_r}_i\}_{i \geq 0} $. On the right, a zoomed-in view of $ \gamma_r $ shows the collections $ \{L^{\gamma_r}_i\}_{i \geq 0} $ and $ \{R^{\gamma_r}_i\}_{i \geq 0} $.}
\label{fig:intervals}
\end{figure}
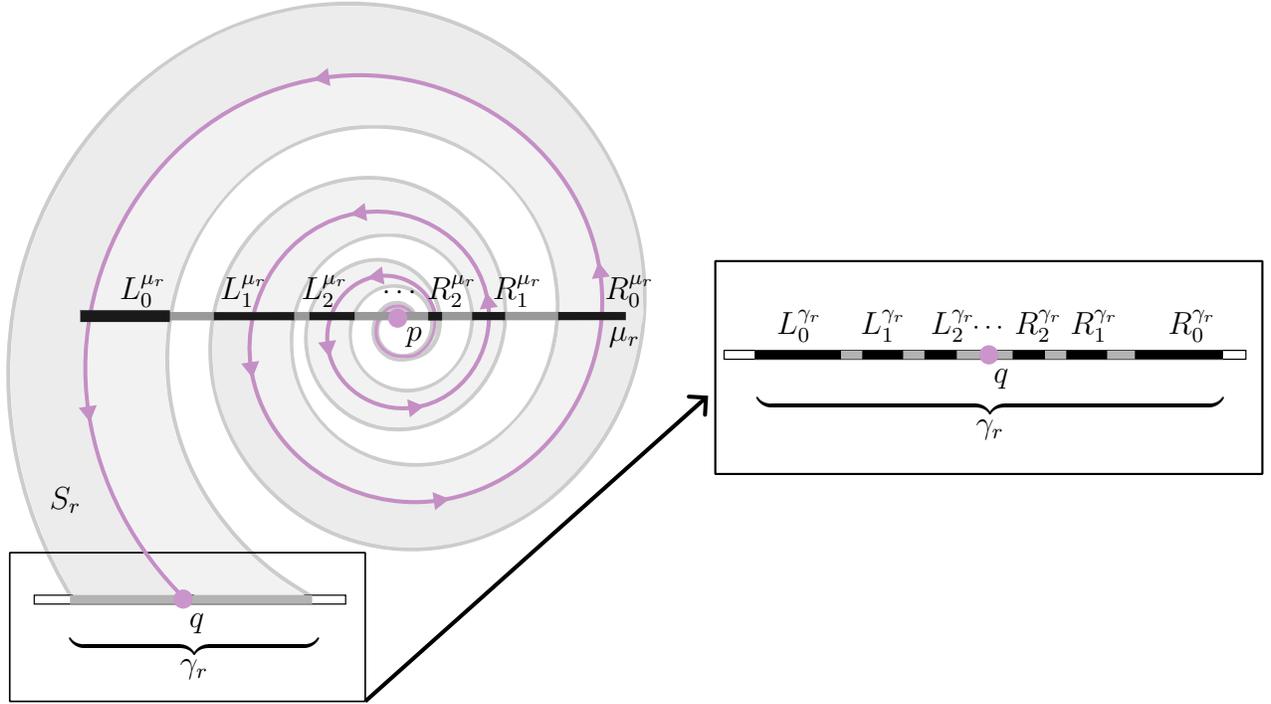

Finally, for each $i \geq 0$, we define
\[ L_i^{\gamma_r} := \theta_X^{-1}(L_i^{\mu_r}) \subset \gamma_r \quad \text{and} \quad R_i^{\gamma_r} := \theta_X^{-1}(R_i^{\mu_r}) \subset \gamma_r, \]
where $\theta_X$ is the diffeomorphism given by \eqref{difeotheta}. This ensures that the properties $L_i^{\gamma_r} \cap L_j^{\gamma_r} = \emptyset$ and $R_i^{\gamma_r} \cap R_j^{\gamma_r} = \emptyset$ hold for $i \neq j$, and that $L_i^{\gamma_r},R_i^{\gamma_r}$ converge to $\{q\}$ in the Hausdorff distance as $i$ increases. Thus, we define
\[ \mathcal{W}_r := \bigcup_{i \geq 0} \left( L_i^{\gamma_r} \cup R_i^{\gamma_r} \right) \subset \gamma_r. \]

Therefore, for $r>0$ sufficiently small, the first return map
\begin{equation}\label{firstreturn}
\begin{aligned}
\pi:\cal{W}_r&\to \gamma_r\\
        w&\mapsto\pi(w) := \varphi^{t_{\tilde{Z}}\left(\theta_X(w)\right)}_{\tilde{Z}}\left(\theta_X(w)\right)
\end{aligned}
\end{equation}
 is well-defined.

\subsection{The First Return Map as a CIFS}
\label{subsec:cifs}

We will now demonstrate that $\pi$ can be branched into countably many smooth maps exhibiting expanding behavior, such that the functions corresponding to their inverses are contractions (see Figure \ref{fig:cifs}) and constitute a CIFS.

Consider the index set $\mathcal{J} = \mathcal{J}_r := \{L_i^{\gamma_r}\}_{i \geq 0} \cup \{R_i^{\gamma_r}\}_{i \geq 0}$. Note that the function $\pi$ maps each $J \in \mathcal{J}$ onto $\gamma_r$, with the possible exceptions of $L_0^{\gamma_r}$ and $R_0^{\gamma_r}$, where $\pi$ may not be surjective. Again, this is not an issue, as we are interested in the local behavior of the map $\pi$. Specifically, we will focus on the restriction of $\pi$ to $U \cap \mathcal{W}_r$, where $U \subset \gamma_r$ is a neighborhood of $q$ satisfying the following condition:
\begin{enumerate}[label= (\Alph*), ref = (\Alph*), start=16]
\item\label{condition:p}
If $J \in \mathcal{J}$ and $J \cap U \neq \emptyset$, then $J \subset U$ and $\pi(J) = \gamma_r$.
\end{enumerate}
For a given neighborhood $U \subset \gamma_r$  of $q$ satisfying \ref{condition:p}, we denote $\cal{J}_U=\{J\in \cal{J}:\, J\subset U\}.$

For each $J\in \cal{J}_U$, the restriction $\pi_{J}:=\pi|_J: J\to \gamma_r$ is a diffeomorphism, thus its inverse $\psi_{J}:=\pi_{J}^{-1} : \gamma_r\to J$ is well-defined. In what follows we will investigate the behavior of $\pi_{J}$ and $\psi_{J}$.

We start by analyzing more deeply the sliding dynamics close to the unstable hyperbolic focus of the sliding vector field $\tilde{Z}$. Since $\mu_R \to \set{p}$ as $R \to 0$, $p$ is an unstable hyperbolic focus of $\tilde Z$, and $\mu_R \setminus \set{p}$ is transverse to $\tilde{Z}$, we can take $R_0 > 0$ such that for any $R \in (0, R_0)$, there exists $\overline{R} > 0$ for which a $\mathcal{C}^{1,1}$ first return map of $\tilde{Z}$ on $\mu_R$, denoted by $\rho: \mu_R \to \mu_{\overline{R}}$, is well-defined. The map $\rho$ is given by $\rho(\zeta) := \varphi^{t_{\rho}(\zeta)}_{\tilde{Z}}(\zeta)$, where $t_{\rho}(\zeta)$ is the time taken for the trajectory of $\tilde{Z}$, starting at $\zeta$, to turn around the focus $p$ and return to $\mu_{\overline{R}}$ on the same side as $\zeta$ relative to $p$. As pointed out by \cite[Proposition 2]{npv}, since $p$ is a hyperbolic unstable fixed point of $\rho$, we can apply the Hartman-Grobman Theorem to get a local $\cal{C}^{1,\beta}$-linearization of $\rho$, for some $\beta>0$, in some neighborhood $\cal{O}\subset\mu_{R_0}$ (see \cite[Theorem 1]{c1linearization} and, also, \cite{hartman,hasselblatt}). That is, there exists a diffeomorphism $H:\cal{O}\to H(\cal{O})$ of class $\cal{C}^{1,\beta}$ with $H(p) = 0$, such that $\rho(\zeta) = H^{-1}(\lambda H(\zeta))$, with $\lambda > 1,$ for every $\zeta\in\mu_R\cap\cal{O}$. Notice that, if $\rho^{k-1}(\zeta)\in \cal{O}$, then $\rho^k(\zeta) = H^{-1}(\lambda^k H(\zeta))$.               
               
Now, in the definition of the first return map \eqref{firstreturn}, take $ r \in (0, R_0) $ such that $ \mu_r \subset \cal{O} $. 
Denote by $\rho_X$ the restriction of $\theta_X$, given by \eqref{difeotheta}, to $\cal{W}_r$, which is a diffeomorphism onto its image $S_r\cap\mu_r$. In addition, given $ w \in \cal{W}_r $, since $ \rho_X(w) \in S_r \cap \mu_r \subset \cal{O} $, we can define an integer $ c(w) $ corresponding to the number of times that $\rho$ must be applied to $ \rho_X(w) $ before it enters $ L_1^{\mu_r}\cup R_1^{\mu_r} $, that is, $$c(w) = i-1\text{ for }w\in L_i^{\gamma_r}\cup R_i^{\gamma_r}.$$
Note that $c(w) = -1$ if $w \in L_0^{\gamma_r}\cup R_0^{\gamma_r}$. However, this is not an issue, as we will consider $w$ in a small neighborhood of $q$. Since the function $c$ is constant in each $J\in \cal{J}$, we denote $c_J:=c(w)$, for $w\in J$, that is, 
	\begin{equation}\label{equation:cJ}
	c_J =  i-1, \text{ if } J\in\{ L_i^{\gamma_r}, R_i^{\gamma_r}\},\quad i\geq0.
 	\end{equation}
	Finally, define the auxiliary map
\[
\rho_{\tilde{Z}}^J:=\begin{cases}
\theta_{\tilde Z}^{L_1^{\mu_r}}(\zeta)&\text{if } J\in\{L_i^{\gamma_r}\}_{i\geq0}, \vspace{0.15cm}\\
\theta_{\tilde Z}^{R_1^{\mu_r}}(\zeta)&\text{if } J\in\{R_i^{\gamma_r}\}_{i\geq0}. \\
\end{cases}
\]	
Notice that, for $J \in \{L_i^{\gamma_r}\}_{i \geq 0}$, we have $\rho_{\tilde{Z}}^J |_{L_1^{\mu_r}} = \theta_{\tilde{Z}}^{L_1^{\mu_r}}$, and for $J \in \{R_i^{\gamma_r}\}_{i \geq 0}$, we have $\rho_{\tilde{Z}}^J |_{R_1^{\mu_r}} = \theta_{\tilde{Z}}^{R_1^{\mu_r}}$. These maps are diffeomorphisms, respectively, from $L_1^{\mu_r}$ and $R_1^{\mu_r}$ onto $\gamma_r$, induced by $\tilde{Z}$ as defined in \eqref{rhoZJ}. It is important to mention that ensuring the diffeomorphic nature of these maps is the reason we chose $L_1^{\mu_r} \cup R_1^{\mu_r}$ to define $c(w)$, rather than $L_0^{\mu_r} \cup R_0^{\mu_r}$.

With the notation introduced above, for $J\in\cal{J}_U$, the restricted first return map $\pi_J$ and its inverse $\psi_J$ are given, respectively, by 
    \begin{equation}\label{equation:psi_j}
    \begin{aligned}
    \pi_{J}(w):=& \rho_{\tilde{Z}}^J\circ\rho^{c_J}\circ\rho_X(w) = \rho^J_{\tilde{Z}}\circ H^{-1}(\lambda^{c_J} H\circ\rho_X (w)),\\
    	\psi_{J}(x):=& \rho^{-1}_X\circ\rho^{-c_J}\circ(\rho_{\tilde{Z}}^J)^{-1}(x) = \rho^{-1}_X\circ H^{-1}\big(\lambda^{-c_J} H\circ(\rho_{\tilde{Z}}^J)^{-1}(x)\big).
\end{aligned}
    	\end{equation}

\begin{proposition}\label{remark:uniformcontractive}
There exists a neighborhood $U\subset\gamma_r$ of $q$ satisfying \ref{condition:p} for which 
$$0 < \epsilon_{J } \leq \abs{\psi'_{J }(x)} \leq s < 1,$$ for every $J\in\cal{J}_U$ and $x \in \gamma_r$. Consequently, $$\abs{\pi'(w)} \geq \dfrac{1}{s}>1,$$ for every $w\in U\cap \cal{W}_r$.\end{proposition}
\begin{proof}
Let $J\in\cal{J}_U$. First, since $\psi_J$ is a diffeomorphism on a compact set $\gamma_r$, then
$$\epsilon_J:=\inf_{x\in\gamma_r}\psi_J'(x)>0.$$

Now, since $\rho_X$, $\rho_{\tilde{Z}}^J$, and $H|_{\mu_r}$ are diffeomorphisms on compact sets, we can set
\begin{equation}\label{equation:A}
A := \left(\min_{\zeta \in J}\abs{(\rho_{\tilde{Z}}^J)'(\zeta)}\right) \cdot \left(\min_{z\in H(\mu_r)}\abs{(H^{-1})'(z)}\right) \cdot \left(\min_{\zeta\in\mu_r}\abs{H'(\zeta)} \right) \cdot \left(\min_{x \in \gamma_r}\abs{\rho_X'(x)}\right) > 0.
\end{equation}
Thus, taking  the expression \eqref{equation:psi_j} into account, we get
\begin{equation}\label{s}
s:=\abs{\sup_{x\in\gamma_r}\psi'_{J }(x)} \leq  \dfrac{1}{A\,\lambda^{c_J}}.
\end{equation}
Hence, we can take a neighborhood $U\subset\gamma_r$  of $q$  satisfying \ref{condition:p} for which $c_J>-\log_{\lambda} A$ for every $J\in\cal{J}_U$. This implies that $s<1$.

Finally, since
\[
\pi_J'(w)=\dfrac{1}{\psi_J'(\pi_J(w))},
\]
we conclude that $\pi_J'(w)\geq s$ for every $w\in J$ and $J\in\cal{J}_U$. Therefore, $\abs{\pi'(w)} \geq \dfrac{1}{s}>1$, for every $w\in U\cap \cal{W}_r$.
\end{proof}

    \begin{remark} \label{remark:attractor}
 The combination of Proposition \ref{remark:uniformcontractive} and the compactness of each set $J \in \cal{J}$ ensures that the collection of functions
\begin{equation*}\label{PsiU}
\Psi_U := \set{\psi_{J} : J \in \cal{J}_U}
\end{equation*}
forms a countable IFS, thereby admitting an attractor set.
    \end{remark}

Before going further, it is important at this point to clarify a subtle detail. Formally, the functions of the IFS $\Psi$ are defined on the curve $\gamma_r$, but working with them in this context can be cumbersome. Instead, we can consider an orientation-preserving diffeomorphism $h: \gamma_r \to [-1,1]$ such that $h(q) = 0$. By defining $\tilde{\psi}_{J} := h \circ \psi_{J} \circ h^{-1}: [-1,1] \to h(J)$, we form the set $\tilde{\Psi} := \big\{\tilde{\psi}_{J} : J \in \cal{J}, J \subset U\big\}$, which is an IFS. Additionally, we have $h(L_i^{\gamma_r}) \subset [-1,0)$ and $h(R_i^{\gamma_r}) \subset (0,1]$. This provides a nice graphical representation of the first return map $\pi$ in Figure \ref{fig:cifs}. Since \ref{hausdorffprop:invariance} implies that the Hausdorff dimension is invariant under bi-Lipschitz maps, which is the case for $h$, this transformation does not change any of the properties we are investigating. Therefore, from now on, we will assume, by an abuse of notation, that the functions in $\Psi$ are defined in the interval $[-1,1]$ and that the sets $J\in\cal{J}$ are subintervals of $[-1,1]$.

\begin{figure}[h!]
\centering
\vspace{0.2cm}
\begin{overpic}[width=\linewidth]{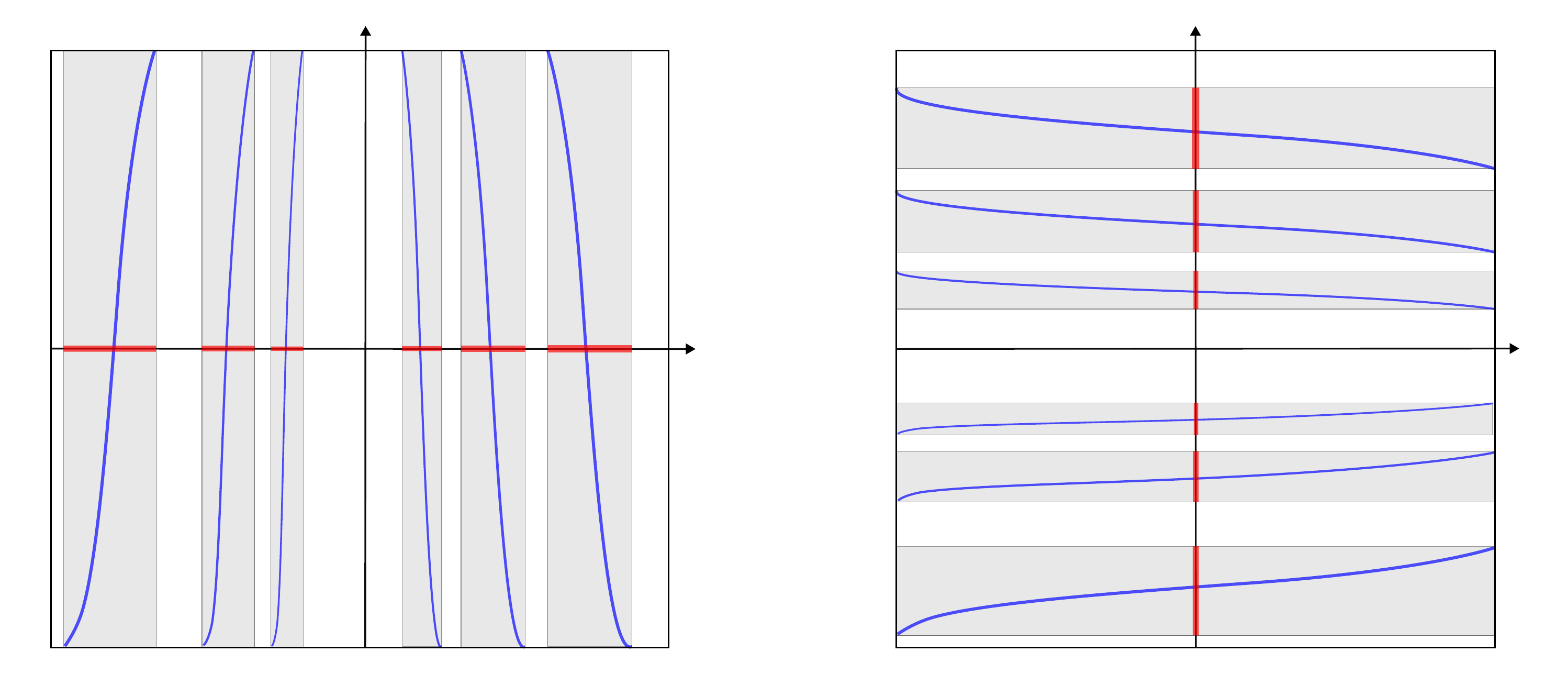}
%\begin{overpic}[grid,tics=2,width=\linewidth]{cifs.pdf}
\put(24,19){\footnotesize $0$}
\put(0,20.2){\footnotesize $-1$}
\put(45,20.2){\footnotesize$1$}
\put(22.5,0){\footnotesize$-1$}
\put(23.8,40.3){\footnotesize $1$}
\put(22.5,42.7){$\gamma_r$}
\put(6,22){$L_1^{\mu_r}$}
\put(13,22){$L_2^{\mu_r}$}
\put(17.5,22){$L_3^{\mu_r}$}
\put(22,22){\footnotesize $\cdots$}
\put(25.5,22){$R_3^{\mu_r}$}
\put(31,22){$R_2^{\mu_r}$}
\put(37,22){$R_1^{\mu_r}$}
\put(6,41){$\pi_{L_1}$}
\put(13,41){$\pi_{L_2}$}
\put(17.5,41){$\pi_{L_3}$}
\put(25.5,41){$\pi_{R_3}$}
\put(31,41){$\pi_{R_2}$}
\put(37,41){$\pi_{R_1}$}

\put(74.5,19){\footnotesize $0$}
\put(53.8,19){\footnotesize $-1$}
\put(95.8,19){\footnotesize $1$}
\put(76,0){\footnotesize $-1$}
\put(77,41){\footnotesize $1$}
\put(98,20.5){$\gamma_r$}
\put(77,3.8){$L_1^{\gamma_r}$}
\put(77,12){$L_2^{\gamma_r}$}
\put(77,15.8){$L_3^{\gamma_r}$}
\put(77,24){$R_3^{\gamma_r}$}
\put(77,28){$R_2^{\gamma_r}$}
\put(77,35){$R_1^{\gamma_r}$}
\put(78,19.8){{\footnotesize $\vdots$}}
\put(96.5,6.3){$\psi_{L_1}$}
\put(96.5,13){$\psi_{L_2}$}
\put(96.5,16.5){$\psi_{L_3}$}
\put(96.5,24.5){$\psi_{R_3}$}
\put(96.5,28.5){$\psi_{R_2}$}
\put(96.5,34){$\psi_{R_1}$}
\end{overpic}
%\vspace{0.2cm}
\caption{Representation of the branches of the first return map $\pi$, starting with $i=1$, and the IFS $\Psi$, respectively.}
\label{fig:cifs}
\end{figure}
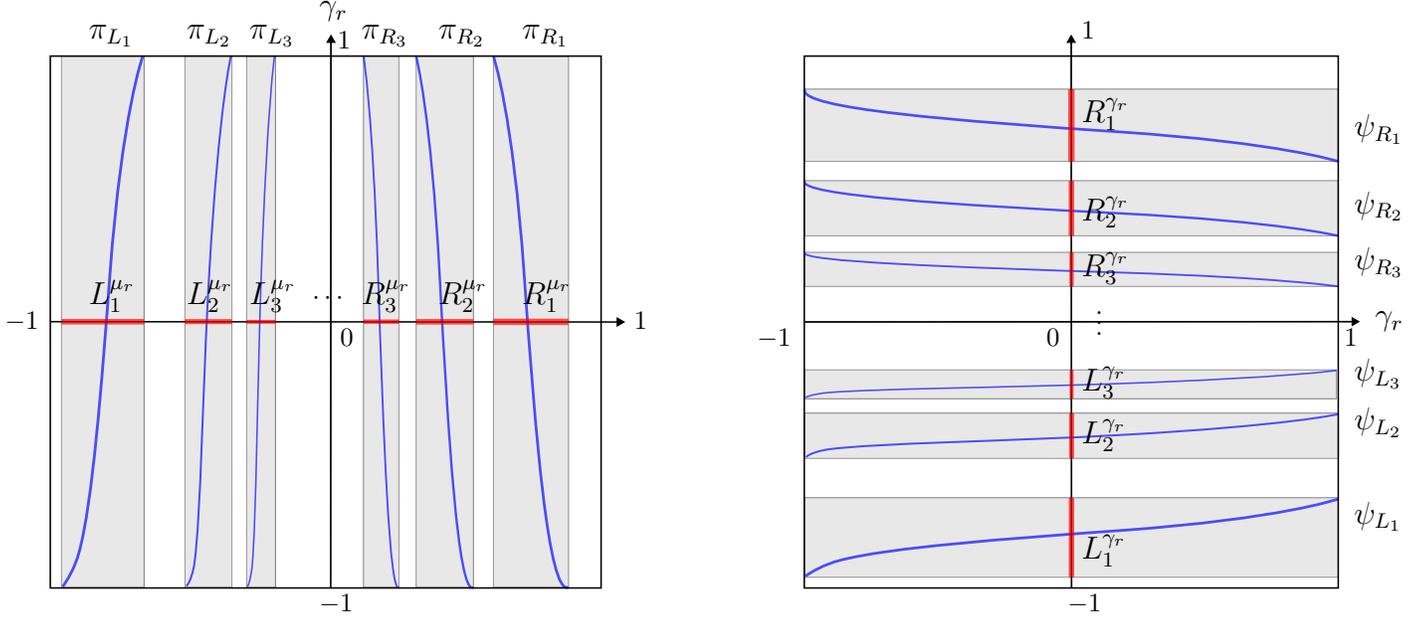

In the following, we will prove that $\Psi_U$ is indeed a CIFS. This is important because it allows the use of Proposition \ref{proposition:switch_union_intersection} to investigate the attractor set of $\Psi_U$, defined in \eqref{atractset}.  For a related recent study on transformations that are countably piecewise differentiable and have countable CIFS in their fibers, as well as on the fiber dimension of their invariant sets, we refer to \cite{M23}.

\begin{proposition}
\label{prop:conformal}
There exists a neighborhood $U\subset\gamma_r$ of $q$ satisfying \ref{condition:p} for which IFS $\Psi_U$ is conformal.
\end{proposition}
\begin{proof}
In what follows, we will examine the conformal conditions individually. Let $U\subset \gamma_r$ be the neighborhood of $q$ given by Proposition \ref{remark:uniformcontractive}. 

We start by the most immediate conditions. First, for  $J\in\cal{J}_U$, $\psi_J$ is a diffeomorphism, in particular, injective so that Condition \ref{conf:injection} holds. Also, Proposition \ref{remark:uniformcontractive} provides $\abs{\psi'_{J }(x)} \leq s < 1$, for every $J\in\cal{J}_U$ and for every $x \in \gamma_r$, which directly implies Condition \ref{conf:unifcontract}. Moreover, $\gamma_r$ is connected and the images of each function $\psi_J$, $J\in\cal{J}_U,$ are mutually disjoint and, therefore, Condition \ref{conf:osc} is satisfied.  Finally, Condition \ref{conf:cone_alt} holds because $K=\gamma_r$ is diffeomorphic to the closed interval $[-1,1]$. It remains to show that Conditions \ref{conf:extension} and \ref{conf:bdp_alt} hold.

Consider the set $\gamma_{r+\delta}$ for $\delta > 0$, and for each $J \in \cal{J}_U$, define the extension $\psi_J^\delta : \gamma_{r+\delta} \to \gamma_{r+\delta}$ of $\psi_J$ by joining affine functions with slopes equal to the lateral derivatives of $\psi_J$ at the endpoints of $\gamma_r$ on each side. This extension is a diffeomorphism, and since $\gamma_{r+\delta}$ is compact, the derivative $(\psi_J^\delta)'$ is Lipschitz continuous and, consequently, $\psi_J^\delta$ is also $\cal{C}^{1,\varepsilon}$. Furthermore, by Proposition \ref{remark:uniformcontractive}, $(\psi_J^\delta)'(x) \neq 0$ for every $x \in \gamma_r$, and therefore, for every $x \in \gamma_{r+\delta}$. Finally, since we are dealing with functions defined on a one-dimensional set, such extensions are trivially conformal, satisfying the requirements for Condition \ref{conf:extension} to hold.

At last, we will show that Condition \ref{conf:bdp_alt} holds. Notice that we have
$$\norm{{\left(\psi'_{J}\right)}^{-1}}_{\opname{unif}}^{-1} = \left(\sup_{x\in\gamma_r}\abs{(\psi'_J(x))^{-1}}\right)^{-1} = \left(\dfrac{1}{\inf_{x\in\gamma_r}\abs{\psi_J'(x)}}\right)^{-1} = \inf_{x\in\gamma_r}\abs{\psi_J'(x)}.$$

By applying the Chain Rule in \eqref{equation:psi_j}, we deduce that
$$
\psi'_{J}(x) = \Big[\left(\rho_X^{-1}\right)'(f_{H^{-1}}(x))\Big] \cdot \Big[\left(H^{-1}\right)'(f_H(x))\Big] \cdot \lambda^{-c_J} \cdot \Big[H'\left(\left(\rho^J_{\Tilde{Z}}\right)^{-1}(x)\right)\Big] \cdot \Big[\left(\left(\rho_{\tilde{Z}}^J\right)^{-1}\right)'(x)\Big],
$$
where $f_{H^{-1}} := H^{-1}\left(\lambda^{-c_J} H\circ\left(\rho_{\tilde{Z}}^J\right)^{-1}\right)$ and $f_H := \lambda^{-c_J} H\circ\left(\left(\rho_{\tilde{Z}}^J\right)^{-1}\right)$. 

Similarly to \eqref{equation:A}, given that each function in the composition is a diffeomorphism, we can set
\begin{equation*}
C:=\bigg(\min_{\zeta\in\mu_r}\abs{\left(\rho_X^{-1}\right)'(\zeta)}\bigg) \cdot \bigg(\min_{z\in H(\mu_r)} \abs{(H^{-1})'(z)}\bigg) \cdot \bigg(\min_{\zeta\in\mu_r} \abs{H'(\zeta)}\bigg) \cdot \bigg(\min_{x\in\gamma_r}\abs{\left(\left(\rho_{\tilde{Z}}^J\right)^{-1}\right)'(x)}\bigg) > 0.
\end{equation*}
This implies that
\begin{equation}\label{equation:inf}
\inf_{x\in\gamma_r}\abs{\psi'_J(x)} \geq C\, \lambda^{-c_J}.
\end{equation}

 Let us denote $T(x) := \lambda^{c_J}\cdot\psi_J(x)$. Notice that $T(x)$ is $\cal{C}^{1,\alpha}$, for some $\alpha>0$, since each of the functions $\rho_{\tilde{Z}}^{-1}$ and $\rho_X^{-1}$ is a $\cal{C}^{1,1}$ diffeomorphism (the same class of differentiability of the fields $X$ and $Y$ that induce them), and $H, H^{-1}$ are $\cal{C}^{1,\beta}$ maps, and in compact domains, products and compositions preserve the H\"{o}lder condition, with some adjustments in the multiplicative and exponential constants (this can be proven using many times the Lipschitz and H\"{o}lder conditions). Keeping this in mind and applying in \eqref{equation:inf}, we deduce the following inequality:
\begin{gather*}
\Bigabs{\abs{\psi'_{J }(x)} - \abs{\psi'_{J }(y)}} \leq \abs{\psi'_{J }(x) - \psi'_{J }(y)} =
\abs{\lambda^{-c_J}\cdot T'(x) - \lambda^{-c_J}\cdot T'(y)} =
\lambda^{-c_J}\cdot\abs{T'(x) - T'(y)} \leq \\
\leq \dfrac{\inf_{\upsilon\in\gamma_r}\abs{\psi'_{J }(\upsilon)}}{C}\cdot D\abs{x-y}^{\alpha} = C^{-1} D \cdot \norm{(\psi'_{J })^{-1}}_{\opname{unif}}^{-1} \cdot \abs{x-y}^{\alpha},
\end{gather*}
for some $D>1, \alpha>0$ and for all $x,y\in\gamma_r$,. We can, then, extend the inequality for $\psi_J^\delta$ and $x,y\in\gamma_{r+\delta}$, since the derivatives $(\psi_J^\delta)'$ in $\gamma_{r+\delta}$ attain values already obtained by $\psi_J'$ in $\gamma_r$ (by construction), so we finally deduce \ref{conf:bdp_alt}.

Therefore, our IFS satisfies all the necessary conditions, making it a CIFS.
\end{proof}

Now, we will demonstrate that the attractor set of the CIFS $\Psi_U$, given by \eqref{atractset}, coincides with the invariant set of the restricted first return map $\pi_U:=\pi|_{\cal{W}_r \cap U}$ given by \eqref{equation:invariantset}.

\begin{proposition}\label{proposition:lambda_equals_delta}
The invariant set $\Lambda_U$ of $\pi_U$ coincides with the attractor set $\Delta_U$ of the CIFS $\Psi_U$.
\end{proposition}
\begin{proof}
First, consider the sets
\begin{equation}\label{Lambdak}
\Lambda_U^k := \bigcap_{0\leq i\leq k}\pi^{-i}\left(\cal{W}_r\cap U\right),\text{ for }k\geq0,\,\, \text{ and }\,\, \Delta_U^k := \bigcup_{\eta\in \cal{J}_U^k}\psi_{\eta}\left(\gamma_r\right),\text{ for } k\geq 1.
\end{equation}
Therefore, the invariant set $ \Lambda_U $ of $\pi_U$, as defined by \eqref{equation:invariantset}, can be expressed as
\begin{equation}\label{charLambda}
\Lambda_U =\bigcap_{k\geq0}\pi^{-k}\left(\cal{W}_r\cap U\right)=\bigcap_{k\geq0}\Lambda_U^k.
\end{equation}
Furthermore, by \Cref{proposition:switch_union_intersection}, the attractor set of  the CIFS $\Psi_U$ is given by
\begin{equation}\label{charDelta}
\Delta_U = \bigcap_{k\geq1}\bigcup_{\eta\in \cal{J}_U^k}\psi_{\eta}\left(\gamma_r\right) = \bigcap_{k\geq1}\Delta_U^k.
\end{equation}

Let us see that the following relationship holds $\Lambda_U^k=\Delta_U^{k+1}$, for every $k\geq 0$. Indeed,
\begin{align*}
\phantom{\iff}\ \ \ x\in \Lambda_U^k 
&\iff \pi^{i}(x)\in J_i\subset \cal{W}_r\cap U\text{ where } J_i\in\cal{J}_U\text{ for }0\leq i\leq k, \text{ and } \pi^{k+1}(x)\in \gamma_r \\
&\iff \pi_{J _k} \circ \pi_{J_{k-1}} \circ \ldots \circ \pi_{J _0}(x) = y, \text{ for some }y\in \gamma_r   \text{ and } J_i\in\cal{J}_U\text{ for }0\leq i\leq k\\
&\iff \psi_{J _k}^{-1} \circ \psi_{J_{k-1}}^{-1}\circ \ldots \circ \psi_{J _0}^{-1}(x) = y, \text{ for some }y\in\gamma_r  \text{ and } J_i\in\cal{J}_U\text{ for }0\leq i\leq k \\
&\iff x = \psi_{J _{0}} \circ \ldots \circ \psi_{J _{k-1}} \circ \psi_{J _k}(y), \text{ for some }y\in\gamma_r  \text{ and } J_i\in\cal{J}_U\text{ for }0\leq i\leq k \\
&\iff x \in \psi_{\eta}\left(\gamma_r\right), \text{ for some }\eta\in \cal{J}_U^{k+1}  \\
&\iff x\in \Delta_U^{k+1}.
\end{align*}
Thus, taking \eqref{charLambda} and \eqref{charDelta} into account, we conclude that $\Lambda_U = \Delta_U$.
\end{proof}

\subsection{The Closure of $\Lambda_U$}
\label{subsection:closure_of_lambda}
Now, we turn our attention to $\cl{\Lambda_U}$, the closure of $\Lambda_U$. The characterization $\Lambda_U=\Delta_U=\bigcap_{k\geq 1} \Delta_U^{k}$ provided by Proposition \ref{proposition:lambda_equals_delta} allows to explore its structure and properties in detail. 
Notice that the sets $\Delta_U^{k}$, $k\geq 1$, as given by \eqref{Lambdak}, satisfy the following recursive property
\begin{equation}\label{equation:recursive_lambda_k}\Delta_U^{k+1} = \bigcup_{\eta\in\cal{J}_U^{k+1}}\psi_\eta(\gamma_r) = \bigcup_{J\in\cal{J}_U}\psi_J(\Delta_U^k), \quad \text{for any $k\geq1$}.
\end{equation}

In order to provide a characterization for the closure $\cl{\Delta_U}$, define 
\begin{equation}\label{QUk}
Q_U^0=\set{q} \quad  \text{ and } \quad Q_U^k := \left(\bigcup_{j=1}^{k}\bigcup_{\eta\in\cal{J}_U^j}\psi_{\eta}(\set{q})\right)\cup \{q\},\text{ for } k\geq1.
\end{equation}
Notice that the set $Q_U^k$, $k\geq0$, also satisfy a recursive property
 \begin{equation}\label{equation:recursive_q_k}
Q_U^{k+1} = \left(\bigcup_{J\in\cal{J}_U}\psi_J(Q_U^k)\right)\cup \{q\}, \quad \text{for any $k\geq0$}.
\end{equation}
Finally, define $$\Delta_U^q := \bigcap_{k\geq1}\left(\Delta_U^k \cup Q_U^{k-1}\right).$$ 

Clearly, since $\Delta_U^k \subset \Delta_U^k \cup Q_U^{k-1}$, we must have $\Delta_U \subset \Delta_U^q$. In the next two propositions, we will show that $\Delta_U^q$ is compact and $\cl{\Delta_U}=\Delta_U^q$.

\begin{proposition}\label{proposition:compacity_of_lambda_q}
The sets $\Delta_U^k \cup Q_U^{k-1}$, for $k\geq 1$, and $\Delta_U^q$ are  compact.
\end{proposition}
\begin{proof}
The proof will be done by induction on $k$.

For $k=1$, we have that $\Delta_U^1\cup Q_U^0 =\left(\cal{W}_r\cap U\right)\cup\set{q}$, which is compact.

Now, assume as inductive hypothesis that the set $\Delta_U^k \cup Q_U^{k-1}$ is compact, for some  $k>1$. In what follows, we will show that $\Delta_U^{k+1} \cup Q_U^{k}$ is also compact.

First, notice that the relationships \eqref{equation:recursive_lambda_k} and \eqref{equation:recursive_q_k} imply
$$\Delta_U^{k+1} \cup Q_U^{k} = \left(\bigcup_{J\in\cal{J}_U}\psi_J(\Delta_U^k)\right) \cup \left(\bigcup_{J\in\cal{J}_U}\psi_J(Q_U^{k-1})\right)\cup \{q\} = \left(\bigcup_{J\in\cal{J}_U}\psi_J (\Delta_U^k\cup Q_U^{k-1})\right) \cup \{q\}.$$
Let $\cal{V}$ be an open cover of $\Delta_U^{k+1} \cup Q_U^{k}$, so in particular it has an open set $V_q \in \cal{V}$ that covers $q$. Denote by $\cal{J}'_U=\{J\in\cal{J}_U: J\subset V_q\}$. Since $L_i^{\gamma_r},R_i^{\gamma_r}$ are converging to $\{q\}$ as $i$ increases, we conclude that $\cal{J}_U\setminus \cal{J}'_U$ is a finite set. Furthermore, since for $J\in\cal{J}_U$, we have $\psi_J(\Delta_U^{k} \cup Q_U^{k-1})\subset J$, it follows that $\psi_J(\Delta_U^{k} \cup Q_U^{k-1})\subset V_q$ for every $J\in  \cal{J}'_U$. Finally, from the inductive hypothesis, we have that
$$\bigcup_{J\in\cal{J}_U\setminus\cal{J}'_U}\psi_J(\Delta_U^{k} \cup Q_U^{k-1})$$ is compact since it is a finite union of compact sets. Therefore, it has finite subcover $\cal{V}'\subset\cal{V}$. Hence,  
$\cal{V}'\cup\set{V_q}\subset\cal{V}$ is a finite subcover of $\Delta_U^{k+1} \cup Q_U^{k}$, which implies that it is compact.

In addition, $\Delta_U^q$ is compact since it is an intersection of compact sets.\end{proof}

\begin{proposition}
\label{proposition:tilde_lambda_equals_closure_lambda}
$\Delta_U^q = \cl{\Delta_U}$.
\end{proposition}
\begin{proof}

First of all, since $\Delta_U\subset\Delta_U^q$, Proposition \ref{proposition:compacity_of_lambda_q} implies that $\cl{\Delta_U}\subset\Delta_U^q$. Thus, it remains to show that $\Delta_U^q\subset\cl{\Delta_U}$ to conclude that $\Delta_U^q = \cl{\Delta_U}$.

Let $x\in\Delta_U^q = \bigcap_{k\geq1}\left(\Delta_U^k \cup Q_U^{k-1} \right)$, that is, $x\in \Delta_U^k \cup Q_U^{k-1}$  for all $k\geq1$. First, if $x\in \Delta_U^k$, for all $k\geq1$, then $x\in\Delta_U\subset\cl{\Delta_U}$ and there is nothing to prove. Thus, suppose that $x\in Q_U^{\kappa}$, for some $0\leq \kappa\leq k-1$. If $\kappa=0$, then $x=q\in\cl{\Delta_U}$. Now, assume that $\kappa\geq 1$. This means that there exists $\ov\eta = (J_1,\ldots,J_{\ell})\in\cal{J}_U^{\ell}$, for some $\ell\leq \kappa$, such that $x = \psi_{\ov\eta}(q) = \psi_{(J_1,\ldots,J_{\ell})}(q)$. Consider a sequence $(x_i)_{i\geq0}$ in $\Delta_U$ converging to $ q$. Since $\psi_{\ov\eta}$ is continuous, we have that
$$x = \psi_{\ov\eta}(q) = \psi_{\ov\eta}\left(\lim_{i\to\infty}x_i\right) = \lim_{i\to\infty}\psi_{\ov\eta}(x_i).$$
Since $\psi_{\ov\eta}(x_i)\in\Delta_U$, because $\Delta_U$ is $\pi$-invariant, we conclude that $x$ the limit of a sequence of elements of $\Delta_U$, therefore $x\in\cl{\Delta_U}$.
\end{proof}

Now, recall that 
\begin{equation}\label{QU}
Q_U = \left(\bigcup_{k\geq0}\pi^{-k}(q)\right)\cap U=\bigcup_{k\geq0}\pi_U^{-k}(q),
\end{equation}
as defined in the statement of Theorem \ref{theorem}. The next proposition details the structure of $\cl{\Lambda_U}$.

\begin{proposition}\label{prop:closurechara}
The set $Q_U$ is countable and $\cl{\Delta_U} = \Delta_U\ \dot{\cup}\ Q_U$ or, equivalently, $\cl{\Lambda_U} = \Lambda_U\ \dot{\cup}\ Q_U$. 
\end{proposition}
\begin{proof}
First of all, we can easily see that
\begin{equation}\label{QUk2}
 Q_U^k= \bigcup_{j=0}^{k}\pi_{U}^{-j}(\set{q}),\text{ for } k\geq 0.
 \end{equation}
This implies that the set $Q_U$, as defined in \eqref{QU}, can be written as
\begin{equation}\label{QU2}
Q_U=\bigcup_{k\geq 0} Q_U^k.
\end{equation}
Notice that, since $\Psi_U$ is a countable set of functions and taking the relationships \eqref{QU2} and \eqref{QUk} into account, we conclude that $Q_U$ is a countable union of countable sets, therefore it is countable. 

In what follows, we proceed with of proof of the equality $\cl{\Delta_U} = \Delta_U\ \dot{\cup}\ Q_U$.

First, we will prove that  $\cl{\Delta_U}\subset  \Delta_U\cup   Q_U$. Let $x\in\cl{\Delta_U} = \bigcap_{k\geq1}\left(\Delta_U^k\cup Q_U^{k-1}\right)$, that is, $x\in\Delta_U^k\cup Q_U^{k-1}$, for all $k\geq1$. If $x\in\Delta_U^k$, for all $k\geq1$, taking \eqref{charDelta} into account, we have that $x\in\Delta_U\subset \Delta_U\cup   Q_U$. Otherwise, for some $\ell\geq 0$ and taking  \eqref{QU2} into account, $x\in Q_U^{\ell}\subset Q_U\subset \Delta_U\cup   Q_U$. Therefore, $\cl{\Delta_U} \subset \Delta_U \cup Q_U$.

Now, for proving the opposite inclusion, $ \Delta_U\cup   Q_U\subset \cl{\Delta_U}$, let $x\in \Delta_U \cup Q_U$. Since $\Delta_U\subset\cl{\Delta_U}$, it only remains to consider the case $x\in Q_U$. Taking \eqref{QU2} into account, let $\kappa\geq0$ 
be the first non-negative integer satisfying  $x\in Q_U^\kappa$.  If $\kappa=0$, then $x = q\in Q_U^k$, for all $k\geq0$, so $x \in\bigcap_{k\geq1}\left(\Delta_U^k\cup Q_U^{k-1}\right) =\cl{\Delta_U}$. On the other hand, assume $\kappa>0$. Notice that, from \eqref{QUk2}, $x\in Q_U^\kappa\subset Q_U^k$ for every $k\geq\kappa$. Thus, let us prove that $x\in \Delta_U^k$ for every $0\leq k\leq\kappa.$ Indeed, from \eqref{QUk}, $x\in Q_U^\kappa$ implies that $x\in\bigcup_{\eta\in\cal{J}_U^\ell}\psi_{\eta}(\set{q})$, for some $1\leq \ell\leq\kappa$. However, if $\ell<\kappa,$ then $x\in Q_U^{\ell}$ contradicting the fact that $\kappa\geq0$ is the first non-negative integer satisfying  $x\in Q_U^\kappa$. Thus, $x\in\bigcup_{\eta\in\cal{J}_U^\kappa}\psi_{\eta}(\set{q})$, that is, $x= \psi_{\ov{\eta}}(q)$, for some $\ov{\eta} = (J_1, \ldots, J_\kappa)\in\cal{J}_U^\kappa$. This means that $x= \psi_{J_1}\circ\cdots\circ\psi_{J_\kappa}(q)\in \Delta_U^k$, for $1\leq k\leq \kappa$. Therefore, $x\in \Delta_U^k\cup Q_U^{k-1}$ for every $k\geq 1$, which implies that $x\in \bigcap_{k\geq1}\left(\Delta_U^k\cup Q_U^{k-1}\right)=\cl{\Delta_U}$.

Furthermore, the union is disjoin, given that $\Delta_U$ is $\pi_U$-invariant and $q\not\in\Delta_U$.
\end{proof}

\subsection{Proof of Theorem \ref{theorem}}
\label{subsec:proof}

\begin{proof}[Proof of \ref{theorem:hausdorff_and_lebesgue}]
Let $U$ be a neighborhood where $\Psi_U$ is a CIFS, whose existence is guaranteed by Proposition \ref{prop:conformal}. Additionally, we can choose $U$ such that 
\begin{equation}\label{cond}
\sum_{J \in \mathcal{J}_U} \dfrac{1}{A \lambda^{c_J}} < 1.
\end{equation} This choice is possible since, from \eqref{equation:cJ}, $c_J =  i-1,$ for $J\in\{ L_i^{\gamma_r}, R_i^{\gamma_r}\}$ and $i\geq0.$

Define
\begin{align*}
P_U(t) := \sum_{J\in\cal{J}_U}\abs{\sup_{x\in\gamma_r}{\psi_{J}'(z)}}^t.
\end{align*}
This function is related to the notion of the \emph{pressure} of an IFS (see \cite[Chapter 3]{mu} and \cite[Chapter 6]{mauldin}). Notice that $P_U(t)$ is positive and unbounded as $t\to 0^+$. In addition, from \Cref{remark:uniformcontractive}, $\abs{\sup_{x\in\gamma_r}{\psi_{J}'(z)}}\leq s<1$, thus $P_U(t)$ is strictly decreasing. Moreover, the relationship \eqref{s}, from the proof of \Cref{remark:uniformcontractive}, says that 
\begin{equation*}\label{cond2}
P_U(t) \leq \sum_{J\in\cal{J}_U}\left(\dfrac{1}{A \lambda^{c_J}}\right)^t,
\end{equation*} 
which, taking \eqref{cond} into account, implies $P_U(1)< 1$. Thus, there exists $0<\tilde{t}<1$ such that $P_U(\tilde{t})=1$.

On the one hand, since $\abs{\psi_J(x) - \psi_J(y)} \leq \sup_{\upsilon \in \gamma_r} \abs{\psi'_J(\upsilon)} \cdot \abs{x - y}$ for each $J \in \cal{J}_U$, \Cref{prop:calculation} and \Cref{prop:sup} imply the upper bound $\hausdorff{\Lambda_U} \leq \tilde{t} < 1$. Consequently, \ref{hausdorffprop:lebesgue_zero} implies $\opname{m}_1(\Lambda_U) = 0$. On the other hand, by applying \Cref{prop:hausdorffzero} and the lower bounds for $\abs{\psi'_J}$ provided by \Cref{remark:uniformcontractive}, we obtain the lower bound $\hausdorff{\Lambda_U} > 0$.
\end{proof}

\begin{proof}[Proof of \ref{theorem:closure_of_lambda}]
Proposition \ref{prop:closurechara} provides that $\cl{\Lambda_U} = \Lambda_U\ \dot{\cup}\ Q_U$, where $Q_U$ is a countable set and, therefore, $\hausdorff{Q_U} = 0$. Thus, \ref{hausdorffprop:stability} and \ref{theorem:hausdorff_and_lebesgue} imply that $\hausdorff{\cl{\Lambda_U}} = \hausdorff{\Lambda_U}<1$. Again, from \ref{hausdorffprop:lebesgue_zero}, we have $\opname{m}_1(\cl{\Lambda_U}) = 0$.

Now we will prove that $\cl{\Lambda_U}$ is a Cantor set, that is a non-empty metric space which is compact, totally disconnected and perfect.

Clearly, $\cl{\Lambda_U}$ inherits the metric structure from $\gamma_r$ and is compact.
Besides, given that $\hausdorff{\cl{\Lambda_U}}<1$, \ref{hausdorffprop:totally_disconnected} says that $\cl{\Lambda_U}$ is a totally disconnected set.
 In order to see that it is perfect, let $x\in\Lambda_U$. From \eqref{atractset}, we know that $x$ is represented by some $\ov{\eta}=(J_1, J_2, J_3, \ldots)\in\cal{J}_U^{\N}$, that is, $x=\proj(\ov{\eta})$, where $\proj$ denotes the projection defined in \eqref{proj}, which is well-defined due to the uniformly contractive property \ref{conf:unifcontract} satisfied by the CIFS $\Psi_U$ (see Remark \ref{rem:proj}). Taking a sequence $\set{x_j}_{j\geq1}\subset\Lambda_U$ satisfying $x_k\in \psi_{(J_1,\ldots,J_k)}(\gamma_r)$, but $x_k\not\in\psi_{(J_1,\ldots,J_k,J_{k+1})}(\gamma_r)$, for $k\geq1$, we have that $\set{x_j}_{j\geq 1}\subset \Lambda_U\setminus\set{x}$ and $x_j\to x$. This implies that every point of $\Lambda_U$ is an accumulation point. Since the closure of a set does not add any isolated points, $\cl{\Lambda_U}$ is a perfect set. This concludes the proof.
\end{proof}

\section*{Acknowledgments}

MGCC is partially supported by Coordena\c{c}\~{a}o de Aperfei\c{c}oamento de Pessoal de N\'{i}vel Superior (CAPES) grant 001. DDN is partially supported by S\~{a}o Paulo Research Foundation (FAPESP) grants 2022/09633-5, 2019/10269-3, and 2018/13481-0, and by Conselho Nacional de Desenvolvimento Cient\'{i}fico e Tecnol\'{o}gico (CNPq) grant 309110/2021-1. GP is partially supported by S\~{a}o Paulo Research Foundation (FAPESP) grants 2022/07762-2 and 2018/13481-0.

\bibliographystyle{abbrv}
\bibliography{references}

\begin{thebibliography}{10}

\bibitem{AC84}
J.-P. Aubin and A.~Cellina.
\newblock {\em Differential inclusions}, volume 264 of {\em Grundlehren der mathematischen Wissenschaften [Fundamental Principles of Mathematical Sciences]}.
\newblock Springer-Verlag, Berlin, 1984.
\newblock Set-valued maps and viability theory.

\bibitem{example1}
G.~Bachar, E.~Segev, O.~Shtempluck, S.~W. Shaw, and E.~Buks.
\newblock Noise-induced intermittency in a superconducting microwave resonator.
\newblock {\em Europhysics Letters}, 89(1):17003, 2010.

\bibitem{example3}
M.~Bernardo, C.~Budd, A.~R. Champneys, and P.~Kowalczyk.
\newblock {\em Piecewise-smooth dynamical systems: theory and applications}, volume 163.
\newblock Springer Science \& Business Media, 2008.

\bibitem{predatorprey}
T.~Carvalho, D.~Duarte~Novaes, and L.~F. Gon{\c{c}}alves.
\newblock Sliding {S}hilnikov connection in {F}ilippov-type predator--prey model.
\newblock {\em Nonlinear Dynamics}, 100(3):2973--2987, 2020.

\bibitem{falconer}
K.~Falconer.
\newblock {\em Fractal geometry}.
\newblock John Wiley \& Sons, Ltd., Chichester, third edition, 2014.
\newblock Mathematical foundations and applications.

\bibitem{filippov}
A.~F. Filippov.
\newblock {\em Differential equations with discontinuous righthand sides}, volume~18 of {\em Mathematics and its Applications (Soviet Series)}.
\newblock Kluwer Academic Publishers Group, Dordrecht, 1988.
\newblock Translated from the Russian.

\bibitem{hasselblatt}
M.~Guysinsky, B.~Hasselblatt, and V.~Rayskin.
\newblock Differentiability of the {H}artman-{G}robman linearization.
\newblock {\em Discrete Contin. Dyn. Syst.}, 9(4):979--984, 2003.

\bibitem{hartman}
P.~Hartman.
\newblock On local homeomorphisms of {E}uclidean spaces.
\newblock {\em Bol. Soc. Mat. Mexicana (2)}, 5:220--241, 1960.

\bibitem{jeffrey14}
M.~R. Jeffrey.
\newblock Hidden dynamics in models of discontinuity and switching.
\newblock {\em Phys. D}, 273/274:34--45, 2014.

\bibitem{jeffrey18}
M.~R. Jeffrey.
\newblock {\em Hidden dynamics}.
\newblock Springer, Cham, 2018.
\newblock The mathematics of switches, decisions and other discontinuous behaviour.

\bibitem{jeffrey20}
M.~R. Jeffrey.
\newblock {\em Modeling with nonsmooth dynamics}, volume~7 of {\em Frontiers in Applied Dynamical Systems: Reviews and Tutorials}.
\newblock Springer, Cham, [2020] \copyright2020.

\bibitem{JDYU22}
M.~R. Jeffrey, T.~I. Seidman, M.~A. Teixeira, and V.~I. Utkin.
\newblock Into higher dimensions for nonsmooth dynamical systems.
\newblock {\em Phys. D}, 434:Paper No. 133222, 13, 2022.

\bibitem{Kivan2011}
V.~K\u{r}ivan.
\newblock On the gause predator–prey model with a refuge: A fresh look at the history.
\newblock {\em Journal of Theoretical Biology}, 274(1):67–73, Apr. 2011.

\bibitem{example5}
O.~Makarenkov and J.~S. Lamb.
\newblock Dynamics and bifurcations of nonsmooth systems: A survey.
\newblock {\em Physica D: Nonlinear Phenomena}, 241(22):1826--1844, 2012.

\bibitem{mauldin}
R.~D. Mauldin.
\newblock Infinite iterated function systems: theory and applications.
\newblock In {\em Fractal geometry and stochastics ({F}insterbergen, 1994)}, volume~37 of {\em Progr. Probab.}, pages 91--110. Birkh\"{a}user, Basel, 1995.

\bibitem{gmw}
R.~D. Mauldin, S.~Graf, and S.~C. Williams.
\newblock Exact {H}ausdorff dimension in random recursive constructions.
\newblock {\em Proc. Nat. Acad. Sci. U.S.A.}, 84(12):3959--3961, 1987.

\bibitem{mu}
R.~D. Mauldin and M.~Urba\'{n}ski.
\newblock Dimensions and measures in infinite iterated function systems.
\newblock {\em Proc. London Math. Soc. (3)}, 73(1):105--154, 1996.

\bibitem{M23}
E.~Mihailescu.
\newblock Invariant measures in non-conformal fibered systems with singularities.
\newblock {\em J. Funct. Anal.}, 284(9):Paper No. 109860, 28, 2023.

\bibitem{mihailescu}
E.~Mihailescu and M.~Urbański.
\newblock Random countable iterated function systems with overlaps and applications.
\newblock {\em Advances in Mathematics}, 298:726--758, 2016.

\bibitem{NJ15}
D.~D. Novaes and M.~R. Jeffrey.
\newblock Regularization of hidden dynamics in piecewise smooth flows.
\newblock {\em J. Differential Equations}, 259(9):4615--4633, 2015.

\bibitem{npv}
D.~D. Novaes, G.~Ponce, and R.~Var\~{a}o.
\newblock Chaos induced by sliding phenomena in {F}ilippov systems.
\newblock {\em J. Dynam. Differential Equations}, 29(4):1569--1583, 2017.

\bibitem{shilnikovproblem}
D.~D. Novaes and M.~A. Teixeira.
\newblock Shilnikov problem in {F}ilippov dynamical systems.
\newblock {\em Chaos}, 29(6):063110, 8, 2019.

\bibitem{pesin_dimension}
Y.~B. Pesin.
\newblock {\em Dimension theory in dynamical systems: contemporary views and applications}.
\newblock University of Chicago Press, 2008.

\bibitem{Piltz2014}
S.~H. Piltz, M.~A. Porter, and P.~K. Maini.
\newblock Prey switching with a linear preference trade-off.
\newblock {\em SIAM Journal on Applied Dynamical Systems}, 13(2):658--682, 2014.

\bibitem{cantor_spaces}
C.~C. Pugh.
\newblock {\em Real Mathematical Analysis}.
\newblock Springer, 2015.

\bibitem{c1linearization}
H.~M. Rodrigues and J.~Sol\`a-Morales.
\newblock Known results and open problems on {$\mathscr{C}^1$} linearization in {B}anach spaces.
\newblock {\em S\~{a}o Paulo J. Math. Sci.}, 6(2):375--384, 2012.

\bibitem{hausdorff_dimension}
D.~Schleicher.
\newblock Hausdorff dimension, its properties, and its surprises.
\newblock {\em The American Mathematical Monthly}, 114(6):509--528, 2007.

\bibitem{example4}
R.~Szalai and M.~R. Jeffrey.
\newblock Nondeterministic dynamics of a mechanical system.
\newblock {\em Physical Review E}, 90(2):022914, 2014.

\bibitem{TEIXEIRA199015}
M.~A. Teixeira.
\newblock Stability conditions for discontinuous vector fields.
\newblock {\em Journal of Differential Equations}, 88(1):15--29, 1990.

\bibitem{vanLeeuwen2013}
E.~van Leeuwen, {\AA}.~Br\"{a}nnstr\"{o}m, V.~Jansen, U.~Dieckmann, and A.~Rossberg.
\newblock A generalized functional response for predators that switch between multiple prey species.
\newblock {\em Journal of Theoretical Biology}, 328:89–98, July 2013.

\end{thebibliography}

\end{document}